\documentclass[11pt]{amsart}
\usepackage{graphicx,amssymb,amsmath,amsthm}
\usepackage{enumerate}
\usepackage{dsfont}
\usepackage{mathrsfs}
\usepackage{epsfig}
\usepackage{lscape}
\usepackage{subfigure}
\usepackage{epstopdf}

\usepackage{caption}
\usepackage{algorithm}
\usepackage{algpseudocode}

\newcommand{\norm}[1]{\|{#1}\|_2}
\newcommand{\normf}[1]{\|{#1}\|_F}
\newcommand{\norms}[1]{\|{#1}\|}
\newcommand{\abs}[1]{\lvert#1\rvert}

\newcommand{\argmin}[1]{\mathop{\rm argmin}\limits_{#1}}
\newcommand{\argmax}[1]{\mathop{\rm argmax}\limits_{#1}}
\newcommand{\minm}[1]{\mathop{\rm min}\limits_{#1}}
\newcommand{\maxm}[1]{\mathop{\rm max}\limits_{#1}}

\newcommand{\barrho}{\bar{\rho}_{i,\alpha}}

\newcommand{\E}{{\mathbb E}}
\newcommand{\PP}{{\mathbb P}}

\newcommand{\1}{{\mathds 1}}

\newcommand{\R}{{\mathbb R}}

\newcommand{\Rnr}{{\mathbb R}^{n \times r}}

\newcommand{\T}{\top}

\newcommand{\Xb}{{\bar{X}}}

\newcommand{\Hh}{{\hat{H}}}
\newcommand{\hh}{{\hat{h}}}
\newcommand{\Ht}{{\tilde{H}}}
\newcommand{\htt}{{\tilde{h}}}

\newcommand{\rank}{{\rm rank}}
\newcommand{\tr}{{\rm tr}}
\newcommand{\ex}{{\rm ex}}

\newtheorem{definition}{Definition}[section]
\newtheorem{corollary}[definition]{Corollary}
\newtheorem{prop}[definition]{Proposition}
\newtheorem{theorem}[definition]{Theorem}
\newtheorem{lemma}[definition]{Lemma}
\newtheorem{claim}[definition]{Claim}
\newtheorem{remark}[definition]{Remark}

\newtheorem{example}[definition]{Example}
\date{}

\begin{document}
\baselineskip 18pt
\bibliographystyle{plain}
\title[Solving Systems of Quadratic Equations]{Solving Systems of Quadratic Equations via Exponential-type Gradient Descent Algorithm}

\author{Meng Huang}
\address{LSEC, Inst.~Comp.~Math., Academy of
Mathematics and System Science,  Chinese Academy of Sciences, Beijing, 100091, China \newline
School of Mathematical Sciences, University of Chinese Academy of Sciences, Beijing 100049, China}
\email{hm@lsec.cc.ac.cn}

\author{Zhiqiang Xu}
\thanks{
      Zhiqiang Xu was supported  by NSFC grant ( 91630203, 11422113,  11331012) and by National Basic Research Program of China (973 Program 2015CB856000).}
\address{LSEC, Inst.~Comp.~Math., Academy of
Mathematics and System Science,  Chinese Academy of Sciences, Beijing, 100091, China \newline
School of Mathematical Sciences, University of Chinese Academy of Sciences, Beijing 100049, China}
\email{xuzq@lsec.cc.ac.cn}

\begin{abstract}
We consider the rank minimization problem from quadratic measurements, i.e., recovering a rank $r$ matrix $X \in \Rnr$ from $m$ scalar measurements $y_i=a_i^\T XX^\T a_i,\;a_i\in \R^n,\;i=1,\ldots,m$. Such problem arises in a variety of applications such as quadratic regression and quantum state tomography. We present a novel algorithm, which is termed {\em exponential-type gradient descent algorithm}, to minimize a non-convex objective function $f(U)=\frac{1}{4m}\sum_{i=1}^m(y_i-a_i^\T UU^\T a_i)^2$. This  algorithm  starts with a careful initialization, and then refines this initial guess by iteratively applying exponential-type gradient descent. Particularly,  we can obtain a good initial guess of $X$
as long as the number of Gaussian random measurements is  $O(nr)$, and  our iteration algorithm can converge linearly to the true $X$ (up to an orthogonal matrix) with $m=O\left(nr\log (cr)\right)$ Gaussian random measurements.
\end{abstract}
\maketitle
\section{Introduction}
\subsection{Problem setup.}
Let $X \in \Rnr$ be a fixed and unknown matrix with ${\rm rank}(X)=r$, and our aim is to recover $X$ from given quadratic measurements, i.e.,
\begin{equation} \label{question}
{\rm find }\quad X\in \R^{n\times r}, \quad {\rm s.t.}\quad
y_i=a_i^\T XX^\T a_i=\norm{a_i^\T X}^2, \qquad i=1,\ldots,m,
\end{equation}
where $a_i=(a_{i,1},\ldots,a_{i,n})\in \R^n$.
This problem is raised in  many emerging applications of science and engineering,
such as  covariance sketching, quantum state tomography and high dimensional data streams   \cite{rankone3,kueng2017low, rankone2}. A simple observation is that $a_i^\T XX^\T a_i=a_i^\T XOO^\T X^\T a_i$ where $O\in \R^{r\times r}$ is an orthogonal matrix. We can only hope to recover $X$ up to a right orthogonal matrix.
There exists an orthogonal matrix $O^*\in \R^{r\times r}$ such that $XO^*$ has orthogonal column vectors.
Hence, throughout the paper we  can assume that $X$ has orthogonal column vectors.

  To recover $X$ from given measurements (\ref{question}),
we consider the following optimization problem:
\begin{equation} \label{optimization problem}
 \minm{U\in \Rnr}  f(U)=\frac{1}{4m}\sum_{i=1}^m(y_i-\norm{a_i^\T U}^2)^2.
\end{equation}
 The aim of this paper is to develop algorithms to solve (\ref{optimization problem}).

 \subsection{Related work}
\subsubsection{Low rank matrix recovery}
 Rank minimization problem is a direct generalization of compressed sensing \cite{recht2010guaranteed,jain2013low}.
 For the general rank minimization problem, it aims to reconstruct a low rank matrix $Q\in \R^{n\times n}$  from incomplete measurements, which can be formulated as the following programming
\begin{equation}\label{eq:minrank}
\begin{array}{l}
 \mathop{\min} \limits_{Z \in \R^{n\times n}} \qquad \rank (Z) \\
  \text{subject to} \quad \tr(A_iZ)=y_i,\quad  i=1,\ldots, m,
\end{array}
\end{equation}
where $y_i=\tr(A_iQ), A_i\in \R^{n\times n}, i=1,\ldots,m$. In \cite{Xulowrank}, Xu has proved that in order to guarantee the solution of (\ref{eq:minrank}) is $Q$ where $Q\in {\mathbb C}^{n\times n}$ and ${\rm rank}(Q)\leq r$,  the minimal measurement number $m$ is $4nr-4r^2$. Since (\ref{eq:minrank}) is non-convex, it is  challenging to solve it \cite{meka2008rank}. However, under a certain restricted isometry property (RIP), this problem can be relaxed to a nuclear norm minimization problem  which is a convex programming and can be solved efficiently  \cite{candes2011tight,recht2010guaranteed}.

Noting that $M:=XX^\T$ is a low rank matrix, we can recast (\ref{question}) as a rank minimization problem. This means that we can use the nuclear norm minimization to recover the matrix $M$ and hence $X$:
\begin{equation}\label{eq:rankone}
\begin{array}{l}
 \mathop{\min} \limits_{Z \in {\mathcal H}_n} \qquad \|Z\|_* \\
  \text{subject to} \quad \tr(A_iZ)=y_i,\quad  i=1,\ldots, m,
\end{array}
\end{equation}
where ${\mathcal H}_n:=\{Q\in \R^{n\times n}:Q=Q^\T\}$ and $A_i=a_ia_i^*$.
The (\ref{eq:rankone}) was studied in \cite{kueng2017low, rankone3}
with proving  that $m\ge Cnr$ Gaussian measurements are sufficient to recover the unknown matrix $M=XX^\T$ exactly. In \cite{rauhut2016low}, Rauhut and Terstiege also consider the case where the measurement vectors $a_i, i=1,\ldots,m$ are from a tight frame.

\subsubsection{Phase retrieval}
Under the setting of $r=1$, the  (\ref{question}) is reduced to phase retrieval problem.
Phase retrieval is to recover an unknown vector from the magnitude of measurements, which means to recover a signal $x\in \mathbb{H}^n$ from measurements
\begin{equation}
 y_i=\abs{\langle a_i,x\rangle}^2, \quad i=1,\ldots,m,
\end{equation}
where $ a_i\in \mathbb{H}^n$ $(\mathbb{H}=\mathbb{C}$ or $\mathbb{R})$ are sampling vectors. This problem is raised  in many imaging applications due to the limitations of optical sensors which can only record intensity information, such as X-ray crystallography \cite{harrison1993phase,millane1990phase}, astronomy \cite{fienup1987phase}, diffraction imaging \cite{shechtman2015phase,gerchberg1972practical}. It has been proved  that $m\ge 4n-4$ Gaussian measurements are sufficient to  recover the unknown vector up to a global phase \cite{phase1}. In recent years, there are several different algorithms have been proposed to solve it \cite{balan2012reconstruction,matrixcompletion,demanet2014stable,eldar2014phase,netrapalli2013phase}.
 In \cite{WF}, Cand\`es et al. design Wirtinger flow algorithm for phase retrieval with solving the following non-convex optimization problem
\begin{equation}\label{eq:gaoxu}
\minm{u\in \mathbb{C}^n} \frac{1}{4m}\sum\limits_{i=1}^m(y_i-\abs{a_i^{*}u}^2)^2
\end{equation}
and prove that the algorithm  converges to the true signal up to a global phase with high probability provided the measurement vectors are  $m=O(n\log n)$ Gaussian measurements.   Following the work of \cite{WF}, Chen and Cand\`es \cite{TWF} propose a modified gradient method which is called {\em Truncated Wirtinger Flow}, and it removes the additional logarithmic factor in the number of measurements $m$. In \cite{Gaoxu}, Gao and Xu propose a Gauss-Newton algorithm to solve (\ref{eq:gaoxu}) and they prove that, for the real signal,  the algorithm can converge to the global optimal solution quadratically with $O(n\log n)$ measurements.

\subsection{Our contribution}
In \cite{thelocal,zheng2015convergent}, one designed algorithms for solving (\ref{optimization problem}). In order to guarantee convergence to the global optimal solution, the algorithm in
\cite{thelocal} requires that $m\ge C\normf{X}^8\lambda_r^{-4}nr^2\log^2n$, while the algorithm in \cite{zheng2015convergent}
needs $m=O(r^3\kappa^2n\log n)$, where $\kappa$ denotes the condition number of $XX^\top$.
 In contrast to those algorithms, we aim to reduce the sampling complexity with removing the additional logarithmic factor on $n$. In this paper, we propose a novel algorithm and call it {\em exponential-type gradient descent algorithm}. For initialization, we give a tighter initial guess through a careful truncated skill; and for iteration update step, we add a moderate bounded exponential-type function to the classical gradient. Particularly, we show the followings all hold with high probability:
\begin{itemize}
\item We present a spectral initial method which obtains a good initial guess provided $m\ge C\sigma_r^{-2}\normf{X}^4nr$ and $a_i, i=1,\ldots,m $ are Gaussian random vectors, where $\sigma_r,\sigma_1$ are the smallest and the largest nonzero eigenvalues of the positive semidefinite matrix $XX^\T $
\item Starting from our initial guess, we refine the initial estimation by iteratively applying a novel gradient update rule. If $m \ge C\sigma_r^{-2}\normf{X}^4nr\log(cr\normf{X}^2/\sigma_r) $, then our algorithm  linearly converges to a global minimizer $X$, up to a right orthogonal matrix. More importantly, the step size in our algorithm is independent with the dimension $n$.
\end{itemize}

\subsection{Organization}
The paper is organized as follows. First, we introduce some  notations and lemmas in Section 2. In Section 3, we introduce the exponential-type gradient descent algorithm  for solving (\ref{optimization problem}).
We study the convergence property of the new algorithm in Section 4. In Section 5, we introduce the main idea for proving the results which given in  Section 4. Numerical experiments are made  in Section 6.
At last, most of the detailed proofs are given in the Appendix.

\section{Preliminaries}
\subsection{Notations}
Throughout the paper, we assume that  $ X =(x_1,\ldots, x_r) \in \Rnr $ has orthogonal columns. Without loss of generality, we  assume that $ \norm{x_1} \geq \norm{x_2} \geq \cdots \geq \norm{x_r}$. We use the Gaussian  random vectors $ a_i\in\R^n, \, i=1,\ldots, m $ as the measurement vectors and obtain $y_i=a_i^\T XX^\T a_i, \, i=1,\ldots,m$. Here we say the sampling vectors are the Gaussian random measurements if $a_i\in {\mathbb R}^n$ are i.i.d. $ \mathcal{N}(0,I) $ random variables.  As we have the entire manifold solutions given by $\mathcal{X}:=\{XO:O\in \mathcal{O}(r)\}$, where $\mathcal{O}(r)$ is the set of $r\times r$ orthogonal matrices, we  define the distance between a matrix $U\in \Rnr$ and $X$ as
\begin{equation}\label{distance}
d(U)\,\,:=\,\,\mathop{\min} \limits_{O \in \mathcal{O}(r)} \|XO-U\|_F.
\end{equation}
To state conveniently, we assume that
\begin{equation} \label{eigen}
\sigma_1 \geq \sigma_2\geq \cdots \geq \sigma_r >0
\end{equation}
are the nonzero eigenvalues of the  matrix $XX^\T $.
\subsection{Lemmas}
We now introduce some lemmas which will be used in our paper. First, we recall a result about random matrix with non-isotropic sub-gaussian rows \cite[Equation (5.26)]{vershynin2010introduction}.
\begin{lemma} (\cite[Equation (5.26)]{vershynin2010introduction}) \label{A introduction}
Let $A$ be an $N\times n$ matrix whose rows are $A_i$, and assume that $\Sigma^{-1/2}A_i$ are isotropic sub-gaussian random vectors, and let $K$ be the maximum of their sub-gaussian norms. Then for every $t\ge 0$, the following inequality holds with probability at least $1-2\exp(-ct^2)$:
\[ \norm{\frac{1}{N}A^*A-\Sigma}\le \max(\delta,\delta^2)\norm{\Sigma} \qquad \text{where} \quad \delta=C\sqrt{\frac{n}{N}}+\frac{t}{\sqrt{N}}.
\]
Here $C,c$ are constants.
\end{lemma}
The next result is Bernstein-type inequality about sub-exponential random variables \cite[Proposition 5.26]{vershynin2010introduction}.
\begin{lemma} (\cite[Proposition 5.26]{vershynin2010introduction}) \label{Bernstein inequality}
Let $X_1,\ldots,X_N$ be independent centered sub-exponential random variables and $K=\max_i\|X_i\|_{\psi_1}$. Then for every $a=(a_1,\ldots,a_N)\in \R^N$ and every $t\ge 0$, we have
\[
\mathbb{P}\Big\{|\sum_{i=1}^Na_iX_i|\ge t\Big\}\le 2\exp\Big[-c\min\big(\frac{t^2}{K^2\norm{a}^2},\frac{t}{K\|a\|_\infty}\big)\Big],
\]
where $c>0$ is an absolute constant.
\end{lemma}
\begin{lemma} \label{lemma 3.1}
For any $\delta>0$, assume that $m\ge 16\delta^{-2}n$ and $a_i, i=1,\ldots,m$ are Gaussian random vectors. Then for any positive semidefinite matrices $M\in \R^{n\times n}$,
\[(1-\delta)\|M\|_*\le \frac{1}{m}\sum_{i=1}^m a_i^\T M a_i \le (1+\delta)\|M\|_* \]
holds on an event $E_\delta$ of probability at least $1-2\exp(-m\epsilon^2/2)$, where $\delta/4=\epsilon^2+\epsilon$ and the norm $\|\cdot\|_*$ denotes the nuclear norm of a matrix. In particular, the right inequality holds for all matrices.
\end{lemma}
\begin{proof}
The first part of this lemma is a direct consequence of  Lemma 3.1 in \cite{phaselift}. Hence, we only need to prove that the right inequality holds for all matrices.
We assume the rank of matrix $M$ is $r$. Then by the singular-value decomposition, we can write $M=\sum_{j=1}^r \sigma_ju_jv_j^\T$, where $u_j,v_j$ are unit vectors. It implies that we just need to show
\begin{equation*}
  \frac{1}{m}\sum_{i=1}^m (a_i^\T u) (a_i^\T v) \le 1+\delta
\end{equation*}
holds for any fixed unit vectors $u,v$. Indeed, if we denote $A:=(a_1,\dots,a_m)^\T$, then
\begin{eqnarray*}
  \sum_{i=1}^m (a_i^\T u) (a_i^\T v) &\le & \frac{1}{2}\sum_{i=1}^m (a_i^\T u)^2+ \frac{1}{2}\sum_{i=1}^m (a_i^\T v)^2 \\
   &=&  (\norm{Au}^2+\norm{Av}^2)/2\\
   &\le & \sigma_{\max}^2(A),
\end{eqnarray*}
where $\sigma_{\max}^2(A)$ is the maximum singular value of $A$. From the well known deviations bounds concerning the singular values of Gaussian random matrices, i.e.,
\begin{equation*}
  \PP(\sigma_{\max}(A)\ge \sqrt{m}+\sqrt{n}+t)\le \exp(-t^2/2),
\end{equation*}
we arrive the conclusion if we take $m\ge \epsilon^{-2}n$ and $t=\sqrt{m}\epsilon$.

\end{proof}
\section{Exponential-type Gradient Descent Algorithm}
Our aim is to recover a  matrix $X\in \Rnr$ (up to right multiplication by an orthogonal matrix) from quadratic measurements
\[
y_i=\norm{a_i^\T X}^2 , \quad i=1,\ldots,m
\]
by solving the non-convex optimization problem
\begin{equation}\label{eq:alg}
 \minm{U\in \Rnr}  f(U)=\frac{1}{4m}\sum_{i=1}^m(y_i-\norm{a_i^\T U}^2)^2.
 \end{equation}
In this section, we will introduce an exponential-type gradient descent algorithm  for solving (\ref{eq:alg}).
\subsection{Spectral Initialization}
 The first step of our algorithm is to choose a good initial guess. In \cite{thelocal},
  Sanghavi, Ward and  White choose $U_0=Z\Lambda^{1/2}$ as the initial guess, where the columns of $Z\in \Rnr$ are the normalized eigenvectors corresponding to the $r$ largest eigenvalues $\lambda_1\ge\cdots\ge\lambda_r$ of the matrix $Y=\frac{1}{2m}\sum\limits_{i=1}^my_ia_ia_i^\T$ and the  diagonal matrix $\Lambda={\rm diag}(\Lambda_1,\ldots,\Lambda_r)$ is given by $\Lambda_i=\lambda_i-\lambda_{r+1}$.  To guarantee the convergence of the iterative method,
 the initialization method introduced in \cite{thelocal} requires $ O(nr^2\log^2 n)$  measurements \cite{thelocal}. Motivated by the methods for choosing the initial guess in \cite{TWF} and \cite{thelocal}, we introduce a novel  initialization method which is stated in Algorithm \ref{initialization1}. We prove that the new method just need $O(nr)$ measurements to obtain the same accuracy as the method suggested in \cite{thelocal}.
\begin{algorithm}[H]
\caption{Initialization}\label{initialization1}
\begin{algorithmic}[H]
\Require
Measurements $ y_i=\|a_i^\T X\|^2, i=1,\ldots,m $, where $a_i$ are Gaussian random vectors; parameter $\alpha_y >0$. \\
Define $ U_0=U\Sigma^{1/2}$, where the columns of $U$ are the normalized eigenvectors corresponding to the $r$ largest eigenvalues $ \lambda_1 \ge \cdots \ge \lambda_{r} $ of the matrix
\[
Y=\frac{1}{m}\sum_{i=1}^{m}y_ia_ia_i^\T \1_{\{y_i \le \frac{\alpha_y}{m}\sum_{k=1}^{m} y_k\}} \]
and  the diagonal matrix $\Sigma$ is given by
\[ \Sigma_{i,i}=\frac{1}{2}(\lambda_i-\lambda_{r+1}).
\]
\Ensure
Initial guess $ U_0 $.
\end{algorithmic}
\end{algorithm}
In our analysis, we require that the parameter $\alpha_y$ in Algorithm 1 satisfies $\alpha_y \ge C\sqrt{\log(c\kappa r)}$, where $\kappa$ is the ratio of the largest to the smallest nonzero eigenvalues of matrix $XX^\T$ and $C,c$ are universal constants. It means that the choice of $\alpha_y$ only depends on  the condition number $\kappa$  and the rank $r$ of $X$,
\subsection{Exponential-type Gradient Descent}
The next step of our algorithm  is to refine the initial guess by  an update rule to search the global optimal solution.
 In \cite{thelocal},  Sanghavi, Ward and  White iteratively update $U$ via gradient descent and they also prove the gradient descent method  converges to the global  optimal solution provided $m\geq Cnr\log^2n$.  We next introduce   an exponential-type gradient descent update rule.

For $k=0,1,\ldots$, we take the iteration step as
\begin{equation}\label{itera step}
 U_{k+1}=U_k-\mu\nabla f_{\ex}(U_k),
\end{equation}
where $ \nabla f_{\ex}(\cdot) $ denotes the exponential-type gradient given by
\begin{equation} \label{gradient}
\nabla f_{\ex}(U)=\frac{1}{m}\sum_{i=1}^m(a_i^\T UU^\T a_i-a_i^\T XX^\T a_i)a_ia_i^\T U \cdot \exp\Big(-\frac{my_i}{\alpha\sum_{k=1}^my_k}\Big),
\end{equation}
where $\alpha>0$.
We state our algorithm as follows:
\begin{algorithm}[H]
\caption{Exponential-type Gradient Descent Algorithm}
\begin{algorithmic}[H]
\Require
Measurement vectors: $a_i\in \R^n, i=1,\ldots,m $; Observations: $y\in \R^m$;  Parameter $\alpha$; Step size $\mu$; $\epsilon>0$  \\
\begin{enumerate}
\item[1:] Set $T:=c\log \frac{1}{\epsilon}$, where $ c $ is a sufficient large constant.
\item[2:] Use Algorithm 1 to compute an initial guess $U_0$ .
\item[3:] For $k=0,1,2,\ldots,T-1$ do
\[\begin{array}{ll}
    U_{k+1} & =U_k-\mu\nabla f_{\ex}(U_k)\\
    \nonumber & =U_k-\frac{\mu}{m}\sum_{i=1}^m(a_i^\T UU^\T a_i-y_i)a_ia_i^\T U \cdot \exp\Big(-\frac{my_i}{\alpha\sum_{k=1}^my_k}\Big)
  \end{array}
\]
\item[4:] End for
\end{enumerate}

\Ensure
The matrix $ U_T $.
\end{algorithmic}
\end{algorithm}

\begin{remark}
There is a parameter $\alpha$ in Algorithm 2.
Throughout this paper, we select the parameter $\alpha\geq 20$.
Numerical experiments in Section 6 show that the algorithm's
performance is not sensitive to the selection of $\alpha$.
\end{remark}

\section{Main results}
In this section we present our main results which give the theoretical guarantee of Algorithm 2.  We first  study Algorithm 1 with showing that our initial guess $U_0$ is
not far from $\{ XO : O \in {\mathcal O}(r) \}$.
\begin{theorem}\label{initial theorem}
Suppose that  $m \ge c_0\sigma_r^{-2}\normf{X}^4nr$  and
\[
y_i=a_i^\T XX^\T a_i=\norm{a_i^\T X}^2,\,\,  i=1,\ldots,m
\]
 where $a_i\in \R^n$ is the Gaussian random vector.
 Let $ U_0 $ be the output of Algorithm \ref{initialization1} with $\alpha_y \ge C\sqrt{\log(c\kappa r)}$, where $\kappa=\sigma_1/\sigma_r$ denotes the ratio of the largest to the smallest nonzero eigenvalues of the matrix $XX^\T$. Then with probability at least $ 1-6\exp(-\Omega(n))$ we have
\[
d(U_0)\,\, \le\,\, \sqrt{\frac{\sigma_r}{8}},
 \]
where $c,c_0$ and $C$ are absolute constants, and $d(U_0)$ is defined as
\[
d(U_0):=\mathop{\min} \limits_{O \in \mathcal{O}(r)} \|XO-U_0\|_F.
\]
\end{theorem}
We next consider the convergence  property of Algorithm 2.
\begin{theorem}\label{itarates}
Suppose that  $m \ge c_0\sigma_r^{-2}\normf{X}^4nr\log(c_1r\normf{X}^2/\sigma_r)$  and
\[
y_i=a_i^\T XX^\T a_i=\norm{a_i^\T X}^2,\,\,  i=1,\ldots,m
\]
 where $a_i\in \R^n$ is the Gaussian random vector. Suppose that $U_k\in \R^{n\times r}$ satisfies  $d(U_k) \le \sqrt{\frac{1}{8}\sigma_r}$. The $U_{k+1}$ is defined by the update rule
 (\ref{itera step}) with the step size $\mu\le \frac{\sigma_r^3}{c_2\sigma_1\normf{X}^6}$.
 Then with probability at least $1-C\exp(-\Omega(n))$, the iteration  step (\ref{itera step}) satisfies
\begin{equation} \label{iterate in theorem}
d(U_{k+1}) \le \Big(1-\rho_0\Big)^{1/2}d(U_k),
\end{equation}
where $\rho_0=\frac{2\mu\sigma_r}{7}$.
\end{theorem}

Combining Theorem \ref{initial theorem} and Theorem \ref{itarates}, we can obtain the following corollary which shows that Algorithm 2 is convergent with high probability provided  $ m\ge Cnr\log (cr)$.

\begin{corollary}
Suppose that  $m \ge c_0\sigma_r^{-2}\normf{X}^4nr\log(c_1r\normf{X}^2/\sigma_r)$  and
$y_i=a_i^\T XX^\T a_i=\norm{a_i^\T X}^2,\,\,  i=1,\ldots,m$ where $a_i\in \R^n$ is the Gaussian random vector.
Suppose that $\epsilon$ is an arbitrary constant within range $(0,\sqrt{\sigma_r/8})$.
Then with probability at least $1-C\exp(-\Omega(n))$, Algorithm 2 outputs $U_T$ satisfying
\[
 d(U_{T})\,\, \le\,\,  \epsilon
\]
provided the step size $\mu\le \frac{\sigma_r^3}{c_2\sigma_1\normf{X}^6}$ where
 $T\ge \log \frac{\sigma_r}{8\epsilon^2}\log \frac{1}{1-\rho_0}$ and $\rho_0=\frac{2\mu\sigma_r}{7}$.
\begin{proof}
According to Theorem \ref{initial theorem}, with probability at least $ 1-6\exp(-\Omega(n))$ we have
\[
d(U_0) \le \sqrt{\frac{\sigma_r}{8}}.
\]
From the iterative inequality (\ref{iterate in theorem}) in Theorem \ref{itarates}, we obtain that
\begin{eqnarray*}
  d(U_{T}) &\le & \Big(1-\rho_0\Big)^{1/2}d(U_{T-1}) \\
   &\le & \Big(1-\rho_0\Big)^{T/2}d(U_{0}) \\
   &\le & \sqrt{\frac{\sigma_r}{8}}\Big(1-\rho_0\Big)^{T/2} \\
   &\le &  \epsilon,
\end{eqnarray*}
which holds with probability at least $1-C\exp(-\Omega(n))$.
\end{proof}
\end{corollary}
\begin{remark}
According to Theorem \ref{itarates}, to guarantee Algorithm 2 converges to the true matrix, we  require that the step size
\begin{equation}\label{eq:step1}
\mu\,\,\le\,\, \sigma_r^3/(C\sigma_1\normf{X}^6).
\end{equation}
Noting that $\normf{X}^4=(\sigma_1+\cdots+\sigma_r)^2\le r^2\sigma_1^2$, we have $\sigma_r^3/(C\sigma_1\normf{X}^6)\ge 1/(C\kappa^3 r^2\normf{X}^2)$ which implies that
\begin{equation}\label{eq:stepsize}
\mu\,\, \le\,\, 1/(C\kappa^3 r^2\normf{X}^2)
\end{equation}
is enough to guarantee (\ref{eq:step1}) holds.
Recall that the algorithms in \cite{thelocal} and \cite{zheng2015convergent} require that
$\mu\le (1/Cn^4\log^4(nr)\normf{X}^2)$ and $\mu\le C/(\kappa n\normf{X}^2)$, respectively.
Comparing with the step size in \cite{thelocal} and \cite{zheng2015convergent}, our  step size is independent with the matrix dimension $n$.
\end{remark}

\section{The proof of the main results}
In this section we give the proof of the main results.  To state conveniently, for $U\in \R^{n\times r}$, we set
\begin{equation}\label{eq:xb}
\Xb:=\Xb_U:=\argmin{Z \in \mathcal{X}}\normf{U-Z},
\end{equation}
 where $\mathcal{X}:=\{XO:O\in \mathcal{O}(r)\}$, and $\mathcal{O}(r)$ is the set of $r\times r$ orthogonal matrices.

Motivated by the results in \cite{WF}, we next give the definition of the regularity condition.
 Under this condition, we shall prove that our algorithm converges linearly to the true matrix $X$ if the initial guess is not far from  it.
\begin{definition}[Regularity Condition]
We say that the function $f$ satisfies the regularity condition $RC(\nu,\lambda,\varepsilon)$ if there exist constants $\nu,\lambda$ such that for all matrices $U \in \Rnr$ satisfying $d(U)\le \varepsilon$ we have
\[
\langle\nabla f_{\ex}(U),U-\Xb\rangle\ge \frac{1}{\nu}\sigma_r\normf{U-\Xb}^2+\frac{1}{\lambda\normf{X}^2}\normf{\nabla f_{\ex}(U)}^2,
 \]
where $ \nabla f_{\ex}(\cdot) $ is defined in (\ref{gradient}) and $\Xb$ is defined in (\ref{eq:xb}).
\end{definition}

Under the assumption of $f$ satisfying the regularity condition, the next lemma shows the performance of the update rule.

\begin{lemma} \label{itera}
Assume that the function $f$ satisfies the regularity condition $RC(\nu,\lambda,\varepsilon)$ and $d(U_k)\le \varepsilon$. If we take the step size $\mu\le \min\left( \frac{\nu}{2\sigma_r}, \frac{2}{\lambda\normf{X}^2}\right)$,
then  $U_{k+1}=U_k-\mu\nabla f_{\ex}(U_k)$ satisfies
\[
d(U_{k+1})\,\,\le\,\, \sqrt{1-\frac{2\mu\sigma_r}{\nu}}d(U_k).
\]
\end{lemma}
\begin{proof}
To state conveniently, we set
\begin{equation}\label{eq:xbk}
\Xb_k:=\argmin{Z \in \mathcal{X}}\normf{U_k-Z}.
\end{equation}
Under the regularity condition $RC(\nu,\lambda,\varepsilon)$, we have
\begin{align}
     d(U_{k+1})^2&\le \normf{U_k-\Xb_k-\mu\nabla f_{\ex}(U_k)}^2 \\
     \nonumber & =\normf{U_k-\Xb_k}^2-2\mu\langle\nabla f_{\ex}(U_k),U-\Xb_k\rangle+\mu^2\normf{\nabla f_{\ex}(U_k)}^2 \\
     \nonumber & \le \normf{U_k-\Xb_k}^2-2\mu\Big(\frac{1}{\nu}\sigma_r\normf{U_k-\Xb}^2+\frac{1}{\lambda\normf{X}^2}\normf{\nabla f_{\ex}(U_k)}^2\Big)
     +\mu^2\normf{\nabla f_{\ex}(U_k)}^2 \\
     \nonumber & =\left(1-\frac{2\mu\sigma_r}{\nu}\right)\normf{U_k-\Xb_k}^2+\mu(\mu-\frac{2}{\lambda\normf{X}^2})\normf{\nabla f_{\ex}(U_k)}^2\\
     \nonumber & \le \left(1-\frac{2\mu\sigma_r}{\nu}\right)d(U_k)^2,
\end{align}
where  the last inequality follows from $\mu\le \frac{2}{\lambda\normf{X}^2}$.
\end{proof}
Based on  Lemma \ref{itera}, the key point to prove Theorem \ref{itarates} is  to show that the function $f$ satisfies the regularity condition with high probability. The next lemma shows that $f$ satisfies the regularity condition provided  $ m \ge c_0\sigma_r^{-2}\normf{X}^4nr\log(c_1r\normf{X}^2/\sigma_r)$.
\begin{lemma}  \label{regularity condition}
Suppose $ m \ge c_0\sigma_r^{-2}\normf{X}^4nr\log(c_1r\normf{X}^2/\sigma_r)$ and $f$ is defined as (\ref{optimization problem}). Then $f$ satisfies the regularity condition $RC\Big(7,\frac{250\alpha^2\sigma_1\normf{X}^4}{\sigma_r^3},\sqrt{\frac{1}{8}\sigma_r}\Big)$ with probability at least $1-C\exp(-\Omega(n))$, where $\alpha$ is the constant in $\nabla f_{\ex}$ and $C,c_0,c_1$ are universal constants.
\end{lemma}

We next state the proof of Theorem \ref{itarates}.

\begin{proof}[Proof of Theorem \ref{itarates}]
According to Lemma \ref{regularity condition}, if $ m \ge c_0\sigma_r^{-2}\normf{X}^4nr\log(c_1r\normf{X}^2/\sigma_r)$, then $f$ satisfies the regularity condition
with $\nu=7$, $\lambda=250\alpha^2\sigma_1\normf{X}^4/\sigma_r^3$ and $\varepsilon=\sqrt{\sigma_r/8}$ with probability at least $1-C\exp(-\Omega(n))$. Noting that $d(U_k) \le \sqrt{\frac{1}{8}\sigma_r}$,  Lemma \ref{itera} implies that
\begin{equation*}
 d(U_{k+1})\,\,\le\,\, \sqrt{1-\frac{2\mu\sigma_r}{\nu}}d(U_k)=\Big(1-\frac{2\mu\sigma_r}{7}\Big)^{1/2}d(U_k)
\end{equation*}
provided the step size
\begin{equation*}
  \mu \le \min\left( \frac{\nu}{2\sigma_r}, \frac{2}{\lambda\normf{X}^2}\right)=\frac{\sigma_r^3}{125\alpha^2\sigma_1\normf{X}^6}=\frac{\sigma_r^3}{c_2\sigma_1\normf{X}^6}.
\end{equation*}
\end{proof}

We remain to prove Lemma \ref{regularity condition}. To this end, we introduce one proposition and the full details can be found in the appendix.
\begin{prop}\label{pr:1}
Assume that $\normf{X}=1$ and that $ m \ge c_0\sigma_r^{-2}nr\log(c_1r/\sigma_r)$. Then with probability at least $1-C\exp(-\Omega(n))$, the followings hold for all matrices  $U \in \Rnr$ satisfying $d(U)\le \sqrt{\frac{\sigma_r}{8}}$:
\begin{eqnarray}
  (a) &\langle\nabla f_{\ex}(U),H\rangle &\ge 0.166\sigma_r\normf{H}^2+0.78\left(\tr^2(H^\T\Xb)+\normf{H^\T\Xb}^2\right) \label{eq:a}\qquad \\
  (b) &\frac{\sigma_r^2\normf{\nabla f_{\ex}(U)}^2}{3\alpha^2\left(\normf{H}^2+\normf{X}^2\right)}&\le 1.223\sigma_1\normf{H}^2+\tr^2(H^\T\Xb)+\normf{H^\T\Xb}^2 \label{eq:b}\qquad,
\end{eqnarray}
where $H=U-\Xb$ and $\Xb$ is defined in (\ref{eq:xb}).
\end{prop}

 Now, we can give the proof of Lemma \ref{regularity condition}.

\begin{proof}[Proof of Lemma \ref{regularity condition} ]
In order to prove Lemma \ref{regularity condition}, we only need to consider the case where $\normf{X}=1$. For any $0<\gamma<1$,  multiplying $\gamma\sigma_r/\sigma_1$ on both sides of  (\ref{eq:b}) we have
\begin{eqnarray*}
  \frac{\gamma\sigma_r^3\normf{\nabla f_{\ex}(U)}^2}{3\alpha^2\sigma_1\left(\normf{H}^2+\normf{X}^2\right)}
   &\le& 1.223\gamma\sigma_r\normf{H}^2+\gamma\sigma_r\tr^2(H^\T\Xb)/\sigma_1+\gamma\sigma_r\normf{H^\T\Xb}^2/\sigma_1.
\end{eqnarray*}
Note that $\sigma_r\leq 1$.
Taking $\gamma=0.166/12.23$ and then combining with $(\ref{eq:a})$, we obtain
\begin{align}
  \langle\nabla f_{\ex}(U),H\rangle
  & \ge 0.1494\sigma_r\normf{H}^2+\frac{\sigma_r^3\normf{\nabla f_{\ex}(U)}^2}{222\alpha^2\sigma_1\left(\normf{H}^2+\normf{X}^2\right)}\nonumber \\
  & \ge 0.1494\sigma_r\normf{H}^2+\frac{\sigma_r^3}{250\alpha^2\sigma_1\normf{X}^2}\normf{\nabla f_{\ex}(U)}^2, \nonumber
\end{align}
where we use $\normf{H}^2\le\frac{1}{8}\sigma_r\le \frac{1}{8}\normf{X}^2$ in the last line. Thus we have
\[\langle\nabla f_{\ex}(U),H\rangle\ge \frac{1}{\nu}\sigma_r\normf{H}^2+\frac{1}{\lambda\normf{X}^2}\normf{\nabla f_{\ex}(U)}^2 \]
for $\nu\ge 7$ and $\lambda\ge 250 \alpha^2\sigma_1/\sigma_r^3 $ with probability at least $1-C\exp(-\Omega(n))$, if $m\ge c_0\sigma_r^{-2}nr\log(c_1r/\sigma_r)$.
\end{proof}

\section{Numerical Experiments}
The purpose of the numerical experiments is the comparison for the exponential-type gradient descent algorithm with the gradient descent algorithm \cite{thelocal}. In our numerical experiments, the target matrix $X\in \Rnr$ is chosen randomly in standard normal distribution and the measurement vector $a_i,i=1,\ldots,m$ are generated by Gaussian random measurements.

\begin{example} \label{example1}
In this example, we test the success rate of the exponential-type gradient descent algorithm with different parameter $\alpha$. Let $X\in \Rnr$ with $n=200,r=2$, the parameter $\alpha_y=9$ in spectral initialization and the step size $\mu=0.1\cdot m/\sum_{i=1}^my_i$.  We consider the performance with $\alpha=20$ and $100$. The maximum number of iterations is $T=3000$. For the number of measurements, we vary $m$ within the range $ [nr,4nr] $. For each $m$, we run 100 times trials and calculate the success rate. We consider a trial to be successful when the relative error is less than $10^{-5}$ and the relative error is defined as
\[   \minm{O\in \mathcal{O}(r)}\frac{\normf{XO-U^t}}{\normf{X}}=\frac{\normf{XZV^\T-U^t}}{\normf{X}}, \]
where $ZDV^\T$ is the singular value decomposition of $X^\T U^t$. Figure \ref{figure:1} shows the numerical results for exponential-type gradient descent and gradient descent algorithm. The figure shows that exponential-type gradient descent algorithm achieve $100\%$ recovery rate if $m\ge 4nr$ and the empirical success rate is better than the gradient descent algorithm.
\end{example}
\begin{figure}[H]
\centering
\includegraphics[width=0.45\textwidth]{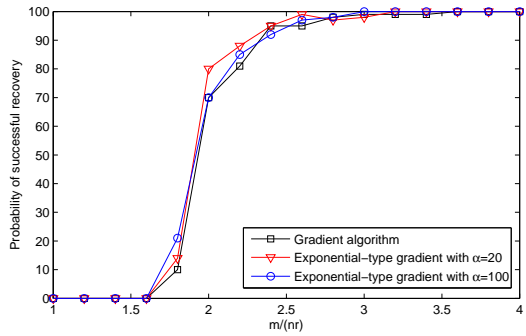}
\caption{Empirical success rate versus $m/nr$ for $X\in \Rnr$ with $n=200$, $r=2$.}
\label{figure:1}
\end{figure}

\begin{example}
In this example, we test the convergence and robustness of the exponential-type gradient descent algorithm.  We use noiseless model for (a) to test the convergence and use the noise model for (b) to test the robustness. The noise model is described as $y_i=a_i^\T XX^\T a_i+\epsilon_i$ where
the noise  $\epsilon_i \sim \mathcal{N}(0,0.1^2),\; i=1,\ldots,m$.
 Let $X\in \Rnr$ with $n=200,r=2$, the parameter $\alpha_y=9$ in spectral initialization and the step size $\mu=0.1\cdot m/\sum_{i=1}^my_i$.  We consider the performance with $\alpha=20$ and $100$. We set the number of measurements $m=3nr$. Figure \ref{figure:2} depicts the relative error against the iteration number. From the figure, we observe that our exponential-type gradient descent algorithm can converge to the exact solution and is robust with noisy measurements.
\end{example}
\begin{figure}[H]
\centering
\subfigure[]{
     \includegraphics[width=0.45\textwidth]{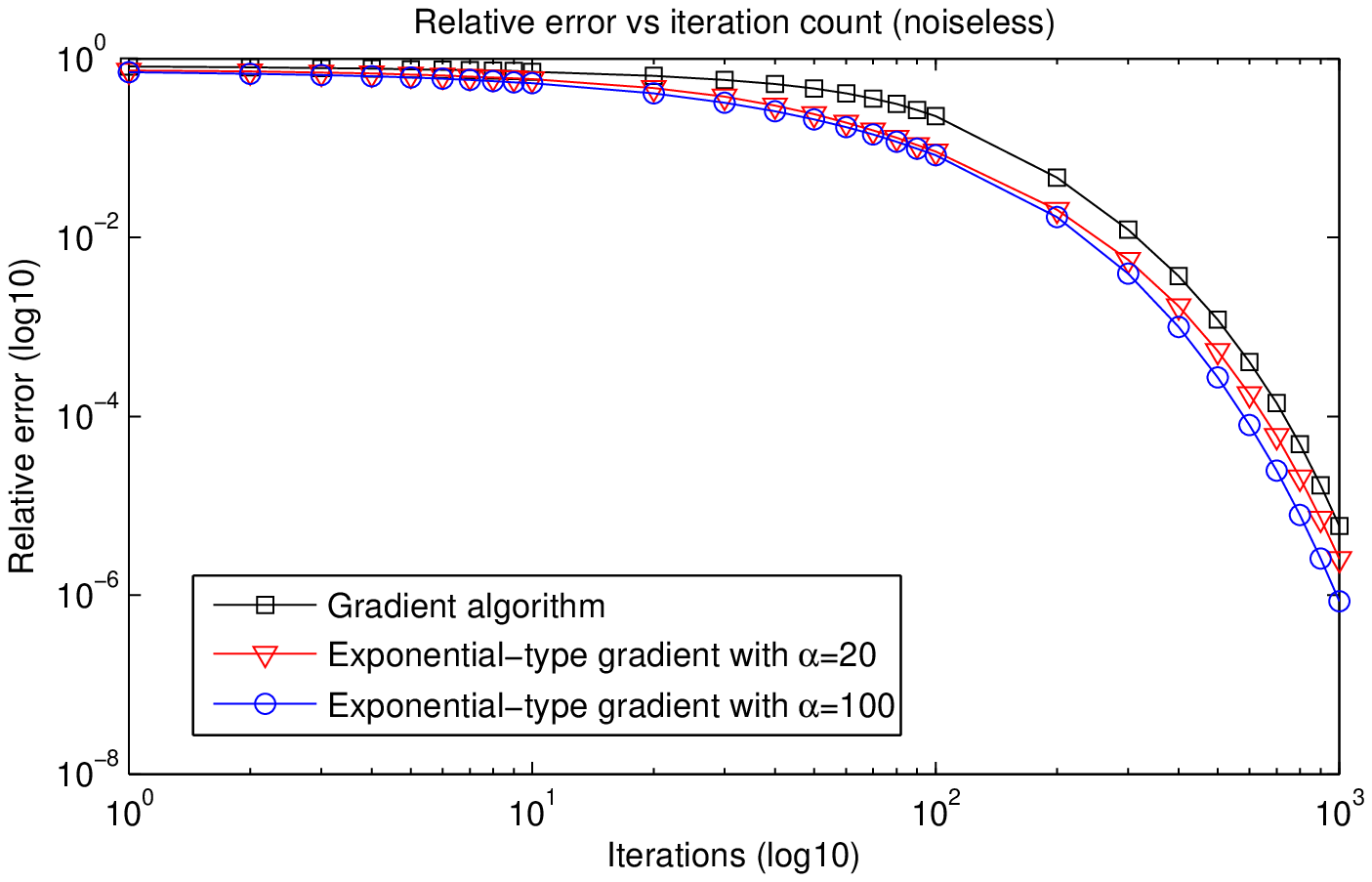}}
\subfigure[]{
     \includegraphics[width=0.45\textwidth]{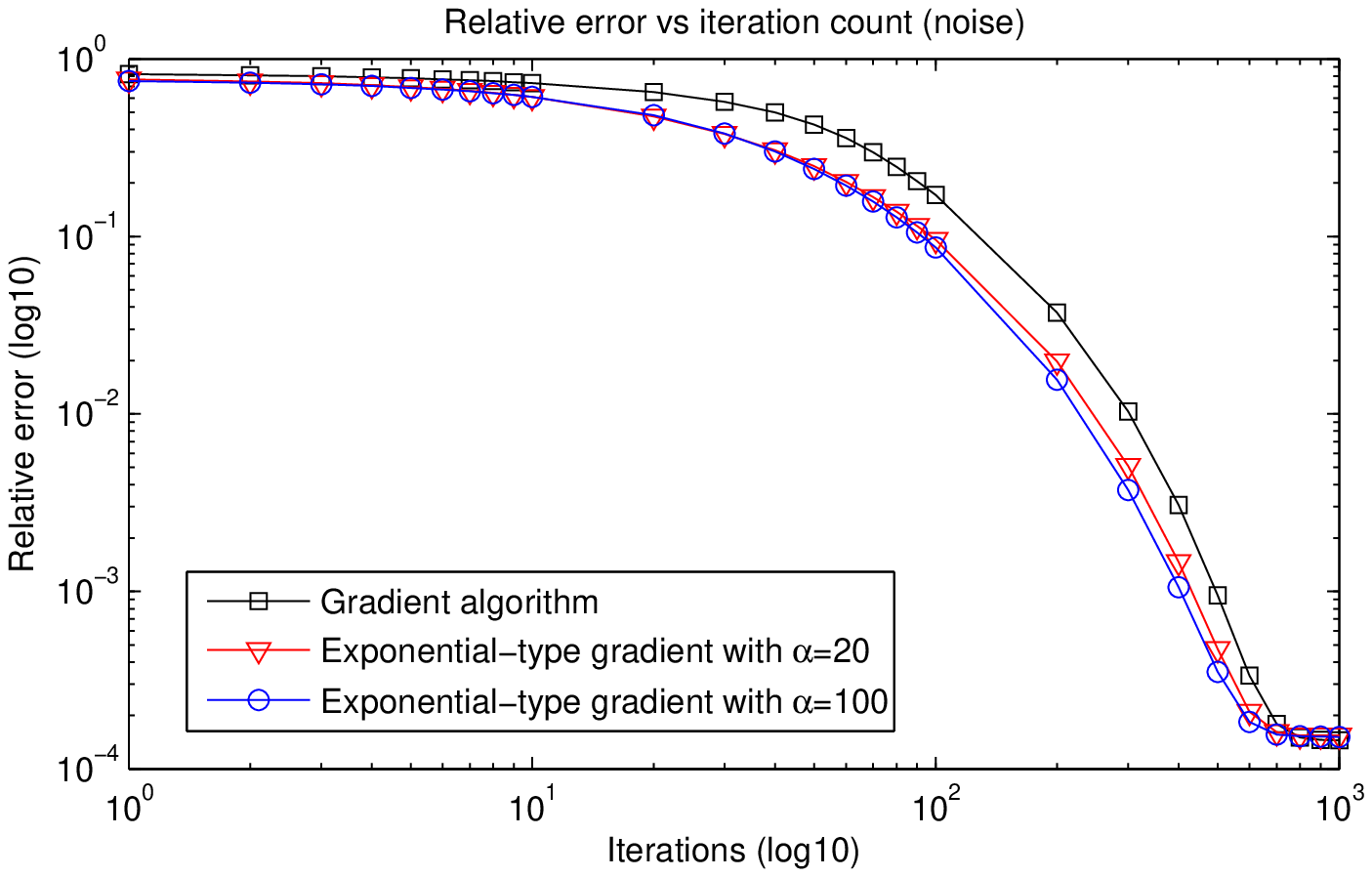}}
\caption{ Relative error with respect to noiseless and noise, where the unknown matrix $X\in \Rnr$ with $n=200$, $r=2$ and $m=3nr$.}
\label{figure:2}
\end{figure}

\begin{figure}[H]
\centering
\subfigure[]{
     \includegraphics[width=0.45\textwidth]{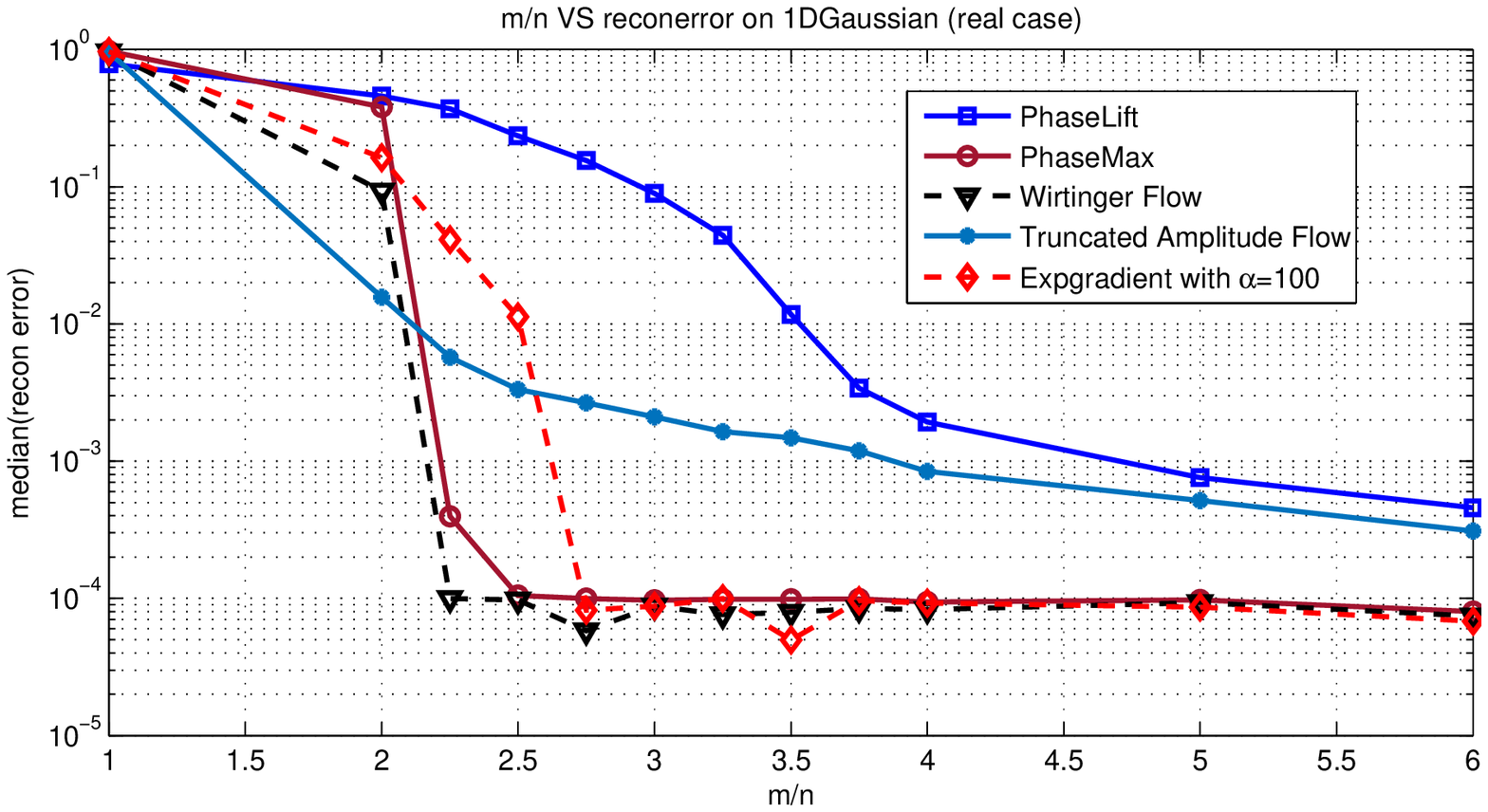}}
\subfigure[]{
     \includegraphics[width=0.45\textwidth]{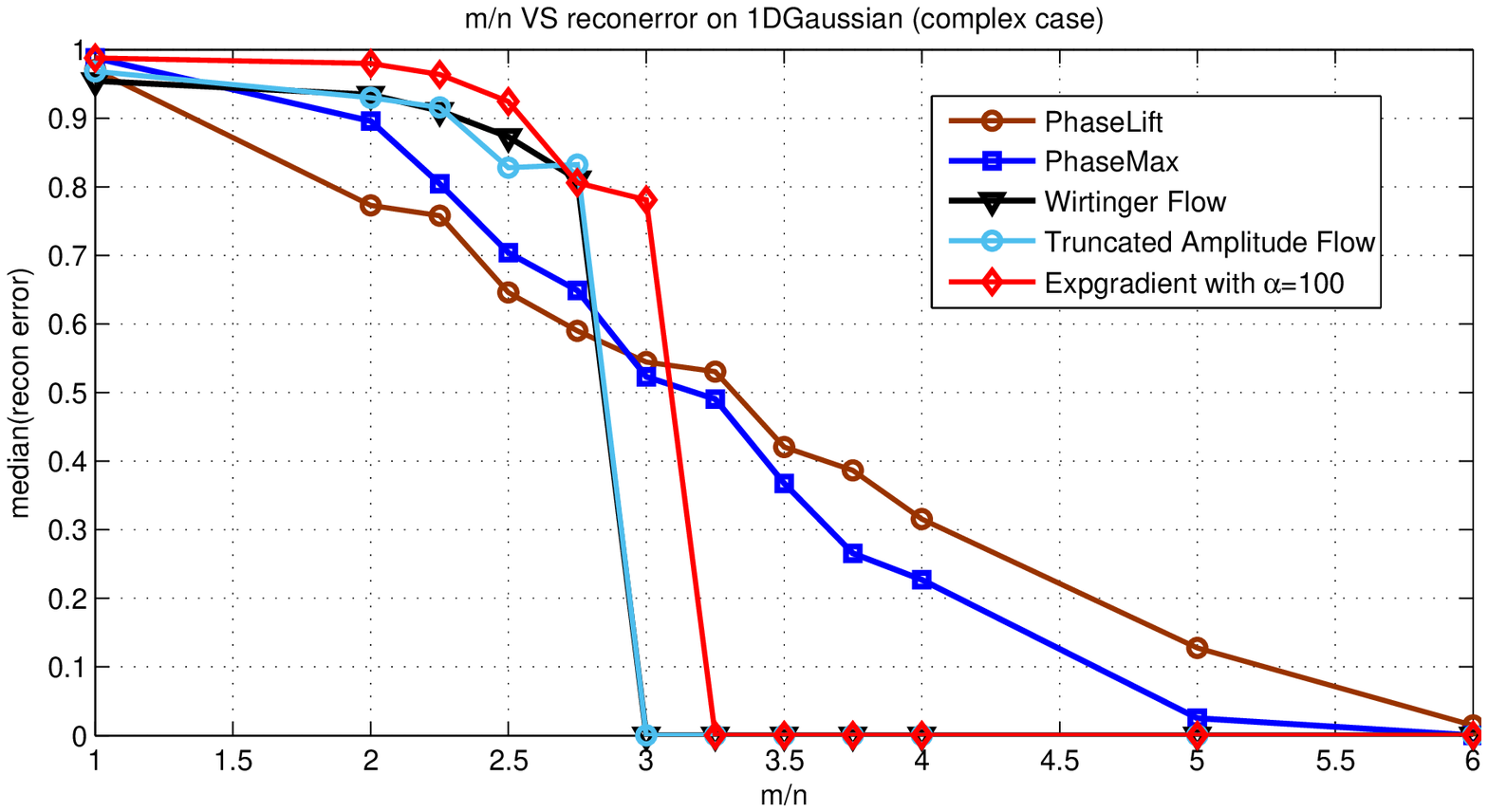}}
\caption{ Relative error versus $m/n$ for real and complex signals $x$ with dimension $n=100$.}
\label{figure3}
\end{figure}
\begin{example}
Finally, we test the performance of the exponential-type gradient descent algorithm to recover $X\in \R^{n\times r}$ with  $r=1$. As stated before, under the setting of $r=1$,
the  (\ref{question}) is reduced to phase retrieval problem. One already develops many algorithms
to solve phase retrieval problems, such as PhaseLift \cite{phaselift}, PhaseMax \cite{goldstein2018phasemax},  WirtFlow \cite{WF}, and TAF \cite{wang2017solving}.
The aim of  numerical experiments  is to compare the performance of the exponential-type gradient descent algorithm with that of other existing methods for phase retrieval as mentioned above.
This experiment is done by Phasepack \cite{phasepack} which is a algorithm package for solving the phase retrieval problem.
 For the exponential-type gradient algorithm, we choose the parameter $\alpha_y=9$ and $\alpha=100$ as the comparison.
We choose a random real signal $x\in \R^n$ in (a) and a random complex signal $x\in \mathbb{C}^n$ in (b) with $n=100$. Here, one can use the elegant formulation of Wirtinger derivatives \cite{WF} to obtain the exponential-type gradient for complex signal. We show the relative error in the reconstructed signal as a function of the the number of measurements $m$, where $m$ within the ranges $[n,6n]$. The results are shown in Figure \ref{figure3}. From the figure, we can see that our algorithm performs well comparing with state-of-the-art phase retrieval algorithms.
\end{example}

\section{Appendix}
\subsection{Proof of Theorem \ref{initial theorem}}
\begin{proof}
By homogeneity, it suffices to consider the case where $\normf{X}=1$. We assume that $X=(x_1,\ldots, x_r) \in \Rnr$ has orthogonal columns satisfying  $ \norm{x_1}\geq \cdots \geq \norm{x_r}$. Recall that $\sigma_1 \geq \sigma_2\geq \cdots \geq \sigma_r >0$ are the nonzero eigenvalues of the positive semidefinite matrix $XX^\T $ and then
\[
\sigma_j=\norm{x_j}^2, \quad \text{ for } \; 1\le j\le r.
\]
From Lemma \ref{lemma 3.1}, for  $\varepsilon>0$,  we have
\begin{equation} \label{average sum y_i}
\frac{1}{m}\sum_{k=1}^ma_k^\top XX^\top a_k = \frac{1}{m}\sum_{k=1}^my_k \in [1- \varepsilon, 1+\varepsilon],
\end{equation}
with probability at least $1-2\exp(-\Omega(n))$,
if $m\ge Cn$ where $C$ is a constant depending on $\varepsilon$. Here, we use the fact that
$\|XX^\top\|_*=\|X\|_F^2=1$.
The (\ref{average sum y_i}) implies that
\begin{equation}\label{eq:jie1}
\1_{\{y_i \le (1-\varepsilon)\alpha_y\}}\le\1_{\{y_i \le \frac{\alpha_y}{m}\sum_{k=1}^{m} y_k\}}\le \1_{\{y_i \le (1+\varepsilon)\alpha_y\}}.
\end{equation}
Recall that $Y=\frac{1}{m}\sum_{i=1}^{m}y_ia_ia_i^\T \1_{\{y_i \le \frac{\alpha_y}{m}\sum_{k=1}^{m} y_k\}}$. The (\ref{eq:jie1}) implies that
\begin{equation}  \label{the interval of Y}
Y_2\preceq Y \preceq Y_1
\end{equation}
holds with  high probability
where
\begin{eqnarray*}
Y_2:=\frac{1}{m}\sum_{i=1}^{m}y_ia_ia_i^\T \1_{\{y_i \le (1-\varepsilon)\alpha_y\}},\quad Y_1:=\frac{1}{m}\sum_{i=1}^{m}y_ia_ia_i^\T \1_{\{y_i \le (1+\varepsilon)\alpha_y\}}.
\end{eqnarray*}
We claim the following results:
\begin{claim}\label{claim of estimate Y1}
 For any $0<\delta<1$, if $\alpha_y \ge C\sqrt{\log(cr\sigma_1/\delta)}$, then
 \begin{equation} \label{estimate of expectation}
  \norm{\E Y_1-2XX^\T-I}\le \delta,\qquad \norm{\E Y_2-2XX^\T-I}\le \delta.
\end{equation}
\end{claim}
The (\ref{estimate of expectation}) implies that $\norm{\E Y_1}\ge 1+2\sigma_1-\delta$ and
$\norm{\E Y_2}\ge 1+2\sigma_1-\delta$.
 We can use  Lemma \ref{A introduction} to obtain that if $m\ge C\delta^{-2}(1+2\sigma_1-\delta)^{-2}n$, and then with probability at least $ 1-4\exp(-\Omega(n))$, we have
\begin{equation} \label{YEY}
 \norm{Y_1-\E Y_1} \le \delta, \qquad \norm{Y_2-\E Y_2} \le \delta,
\end{equation}
where $C$ is a  positive constant.
Indeed, in Lemma \ref{A introduction} we take the $i$-th row of $A$ as  $b_i^\T:=\sqrt{y_i}a_i^\T \1_{\{y_i \le (1+\varepsilon)\alpha_y\}}$ and set $\Sigma=\E Y_1$ with $\norm{\E Y_1}\ge 1+2\sigma_1-\delta$ and $t=\delta\norm{\E Y_1}\sqrt{m}$. Then we can  obtain $\norm{Y_1-\E Y_1} \le \delta$. Similarly, we have $\norm{Y_2-\E Y_2} \le \delta$ if we take the $i$-th row of $A$ as  $b_i^\T:=\sqrt{y_i}a_i^\T \1_{\{y_i \le (1-\varepsilon)\alpha_y\}}$ and set $\Sigma:=\E Y_2$.

Combining (\ref{the interval of Y}), (\ref{estimate of expectation}) and (\ref{YEY}), we have
\begin{equation}\label{approximate of Y}
 \norm{Y-2XX^\T-I}\le 2\delta
\end{equation}
with probability at least $1-6\exp(-\Omega(n))$ provided $m\ge C\delta^{-2}(1+2\sigma_1-\delta)^{-2}n$ and $\alpha_y \ge C\sqrt{\log(cr\sigma_1/\delta)}$. Furthermore, from Wely Theorem we have
\begin{equation}\label{wely}
|\lambda_{r+1}-1|\le 2\delta \quad \text{and} \quad |\lambda_{n}-1|\le 2\delta.
\end{equation}

Next, we turn to consider $d(U_0)$. Recall the definition $U_0=U\Sigma^{1/2}$ in Algorithm \ref{initialization1}. Here, $U=(u_1,\ldots,u_r)$ where $u_k$ is normalized eigenvectors corresponding to the eigenvalues $\lambda_k$ of $Y$ for $k=1,\ldots,r$, and the scaling of the diagonal matrix $\Sigma$ is given by $\Sigma_{i,i}=(\lambda_i-\lambda_{r+1})/2$.
Hence,
\begin{eqnarray*}
  \norm{U_0U_0^\T-XX^\T} &\le & \norm{U_0U_0^\T-\frac{1}{2}Y+\frac{1}{2}\lambda_{r+1}I}+\norm{\frac{1}{2}Y-\frac{1}{2}I-XX^\T}+\frac{1}{2}\norm{(\lambda_{r+1}-1)I} \\
   &\le& \frac{1}{2}(\lambda_{r+1}-\lambda_n)+\delta+\frac{1}{2}(\lambda_{r+1}-1) \\
   &\le& 4\delta ,
\end{eqnarray*}
where the second inequality follows from (\ref{approximate of Y}) and the last inequality follows from (\ref{wely}). Then, using the following fact ( see, e.g. the Initialization of \cite{zheng2015convergent})
\begin{eqnarray*}
  \mathop{\min} \limits_{O \in \mathcal{O}(r)} \|U_0-XO\|_F^2 &\le & \frac{\normf{U_0U_0^\T-XX^\T}^2}{(2\sqrt{2}-2)\sigma_r},
\end{eqnarray*}
and taking $\delta\le \frac{\sigma_r}{18\sqrt{r}}$, we obtain
\begin{eqnarray*}
  \mathop{\min} \limits_{O \in \mathcal{O}(r)} \|U_0-XO\|_F^2 &\le & \frac{2r\norm{U_0U_0^\T-XX^\T}^2}{(2\sqrt{2}-2)\sigma_r} \\
   &\le & \frac{32r\delta^2}{(2\sqrt{2}-2)\sigma_r} \\
   &\le& \frac{\sigma_r}{8},
\end{eqnarray*}
where we use $\normf{A}\le\sqrt{\rank(A)}\norm{A}$ in the first inequality. The choice of $\delta$ implies that the measurements $m\ge C\sigma_r^{-2}nr$ and $\alpha_y \ge C\sqrt{\log(c'\kappa r)}$, where $\kappa=\sigma_1/\sigma_r$ denotes the ratio of the largest to the smallest nonzero eigenvalues of matrix $XX^\T$.\\

We remain to prove  Claim \ref{claim of estimate Y1}.
There exists an orthogonal matrix $O\in \R^{r\times r}$ such that $X=O(\norm{x_1}e_1,\ldots,\norm{x_r}e_r)$. Then
\begin{eqnarray*}
  O^\T(\E Y_1-2XX^\T-I)O &=& O^\T \E Y_1 O-\left(2\sum_{k=1}^r\norm{x_k}^2e_ke_k^\T +I \right),
\end{eqnarray*}
and
\begin{eqnarray} \label{Y_1}
      O^\T \E Y_1 O &=& \E \left[\sum_{k=1}^r \norm{x_k}^2a_{i,k}^2a_ia_i^\T \1_{\{\sum_{k=1}^r \norm{x_k}^2a_{i,k}^2\le(1+\varepsilon)\alpha_y\}} \right].
 \end{eqnarray}
A simple calculation is that
\begin{equation}\label{eq:pr41}
  \E \left[\sum_{k=1}^r \norm{x_k}^2a_{i,k}^2a_ia_i^\T\right] = 2\sum_{k=1}^r\norm{x_k}^2e_ke_k^\T +I ,
\end{equation}
which implies that
\begin{equation} \label{upper bound}
O^\T \E Y_1 O \le 2\sum_{k=1}^r\norm{x_k}^2e_ke_k^\T +I,
\end{equation}
where we write $M_2\le M_1$ if all entries of $M_1-M_2$ are nonnegative. On the other hand, from (\ref{Y_1}) we obtain that
\begin{equation}\label{eq:pr42}
  O^\T \E Y_1 O = \E \left[\sum_{k=1}^r \norm{x_k}^2a_{i,k}^2a_ia_i^\T\right]-\E \left[\sum_{k=1}^r \norm{x_k}^2a_{i,k}^2a_ia_i^\T \1_{\{\sum_{k=1}^r \norm{x_k}^2a_{i,k}^2\ge(1+\varepsilon)\alpha_y\}} \right].
\end{equation}
For any $1\le j,l,k\le r$ and  $\delta>0$, by H\"{o}lder's inequality we have
\begin{equation}\label{eq:bernsteinfordetal}
\begin{aligned}
  & \E\left[\norm{x_k}^2a_{i,j}^2a_{i,l}^2 \1_{\{\sum_{k=1}^r \norm{x_k}^2a_{i,k}^2\ge(1+\varepsilon)\alpha_y\}} \right] \\
& \quad  \le  \norm{x_1}^2\sqrt{\E[a_{i,j}^4a_{i,l}^4]}\cdot \sqrt{\PP\left\{\sum_{k=1}^r \norm{x_k}^2a_{i,k}^2\ge(1+\varepsilon)\alpha_y\right\}} \\
   &\quad \le C_1\norm{x_1}^2\exp\left(-C_0\min\Big(\frac{(1+\varepsilon)^2\alpha_y^2}{\norm{x_1}^4
+\cdots+\norm{x_r}^4},\frac{(1+\varepsilon)\alpha_y}{\norm{x_1}^2}\Big)\right) \\
   &\quad \le C_1\sigma_1\exp\left(-C_0(1+\varepsilon)^2\alpha_y^2\right) \\
   &\quad \le \frac{\delta}{r}
\end{aligned}
\end{equation}
provided $\alpha_y \ge C\sqrt{\log(cr\sigma_1/\delta)}$, where  the second inequality  follows from Lemma \ref{Bernstein inequality} and  the third inequality  follows from the fact that $\normf{X}=1$ and $\norm{x_r}\le \cdots \le \norm{x_1}\le 1$. The (\ref{eq:bernsteinfordetal}) implies that
\begin{equation}\label{eq:pr43}
\E \left[\sum_{k=1}^r \norm{x_k}^2a_{i,k}^2a_ia_i^\T \1_{\{\sum_{k=1}^r \norm{x_k}^2a_{i,k}^2\ge(1+\varepsilon)\alpha_y\}} \right]\le \delta I.
\end{equation}
Thus, combining (\ref{eq:pr41}), (\ref{eq:pr42}) and (\ref{eq:pr43}) we have
\begin{eqnarray}\label{lower bound}
  O^\T \E Y_1 O &\ge 2\sum\limits_{k=1}^r\norm{x_k}^2e_ke_k^\T +(1-\delta)I.
\end{eqnarray}
Combining   (\ref{upper bound}) and (\ref{lower bound}) and noting that $O^\T \E Y_1 O$ is a diagonal matrix, we obtain
\begin{eqnarray*}
    \norm{\E Y_1-2XX^\T-I} &=& \norm{O^\T(\E Y_1-2XX^\T-I)O} \le \delta.
\end{eqnarray*}
Similarly, we can obtain $\norm{\E Y_2-2XX^\T-I}\le \delta$, which completes the proof.
\end{proof}

\subsection{Proof of Proposition \ref{pr:1} }
We always assume that $\normf{X}=1$ throughout the proof.
We set $H:=U-\Xb$ where $\Xb=\argmin{Z \in \mathcal{X}}\normf{U-Z}$ and $\mathcal{X}$ is the solution set. Then the exponential-type gradient can be rewritten as
\begin{equation}\label{regradient}
 \nabla f_{\ex}(U)=\frac{1}{m}\sum_{i=1}^m(a_i^\T HH^\T a_i+2a_i^\T H\Xb^\T a_i)(a_ia_i^\top H+a_ia_i^\T\Xb) \cdot \exp\left(-\frac{my_i}{\alpha\sum_{k=1}^my_k}\right).
\end{equation}
For convenience, we let
\begin{equation}\label{hi}
  \rho_{i,\alpha}\,\,:=\,\,\exp\big(-\frac{my_i}{\alpha\sum_{i=1}^my_i}\big),\;i=1,\ldots,m.
\end{equation}
To prove Proposition \ref{pr:1}, we need the following lemmas.

\begin{lemma} \label{lemma one}
For any fixed $\alpha\ge 20$ and $\delta>0$, if $ m\ge c_0\alpha^2\delta^{-2}nr\log(\sqrt{r}/\delta)$, then with probability at least $1-C\exp(-\Omega(\alpha^{-2}\delta^2m))$, the followings hold for all non-zero matrix $U \in \Rnr$:
\begin{eqnarray*}
  (a)&\frac{1}{m}\sum_{i=1}^m(a_i^\T H\Xb^\T a_i)^2 \rho_{i,\alpha}\ge& (0.78\sigma_r-2\delta)\normf{H}^2+0.78\tr^2(H^\T\Xb)+0.78\normf{H^\T\Xb}^2 \\
  (b)&\frac{1}{m}\sum_{i=1}^m(a_i^\T H\Xb^\T a_i)^2 \rho_{i,\alpha}\le& (\sigma_1+2\delta)\normf{H}^2+\tr^2(H^\T\Xb)+\normf{H^\T\Xb}^2,
\end{eqnarray*}
where $C,c_0$ are universal constants.
\end{lemma}
\begin{proof}
Suppose for the moment that $H$ is independent from $a_i$. By homogeneity, it suffices to establish the claim for the case $\normf{H}=1$. From (\ref{average sum y_i}) we have
\begin{equation}\label{ei}
\exp\Big(-\frac{a_i^\T XX^\T a_i}{0.99\alpha}\Big)\le  \rho_{i,\alpha}\le \exp\Big(-\frac{a_i^\T XX^\T a_i}{1.01\alpha}\Big)
\end{equation}
with high probability. For convenience,  we set
\begin{equation}\label{li}
  \barrho:=\exp\Big(-\frac{a_i^\T \Xb\Xb^\T a_i}{0.99\alpha}\Big),\;i=1,\ldots,m.
\end{equation}
Noting that $a_i^\T \Xb\Xb^\T a_i= a_i^\T XX^\T a_i$, we have
 \begin{equation}\label{eq:1le52}
 \frac{1}{m}\sum_{i=1}^m(a_i^\T H\Xb^\T a_i)^2\rho_{i,\alpha}\ge \frac{1}{m}\sum_{i=1}^m(a_i^\T H\Xb^\T a_i)^2\barrho.
\end{equation}
We claim the following results:
\begin{claim}\label{claim1}
For any fixed parameter $\alpha\ge 20$ it holds
\begin{itemize}
  \item[1)] $\E\left[(a_i^\T H\Xb^\T a_i)^2\right]\ge\sigma_r\normf{H}^2+\tr^2(H^\T\Xb)+\normf{H^\T\Xb}^2$\vspace{1ex}
  \item[2)] $\E\left[(a_i^\T H\Xb^\T a_i)^2\right]\le\sigma_1\normf{H}^2+\tr^2(H^\T\Xb)+\normf{H^\T\Xb}^2$\vspace{1ex}
  \item[3)] $\E\left[(a_i^\T H\Xb^\T a_i)^2\barrho\right]\ge 0.78\E\left[(a_i^\T H\Xb^\T a_i)^2\right]$.\vspace{1ex}
\end{itemize}
\end{claim}
Then combining 3) and 1) we obtain that
\begin{eqnarray*}
    \E\left[(a_i^\T H\Xb^\T a_i)^2\barrho\right] &\ge& 0.78\sigma_r\normf{H}^2+0.78\tr^2(H^\T\Xb)+0.78\normf{H^\T\Xb}^2.
\end{eqnarray*}
Since
\begin{eqnarray*}
  (a_i^\T H\Xb^\T a_i)^2\barrho &\le & (a_i^\T \Xb\Xb^\T a_i)\barrho(a_i^\T HH^\T a_i)
\end{eqnarray*}
 and $(a_i^\T \Xb\Xb^\T a_i)\barrho$ is bounded, it means that $(a_i^\T H\Xb^\T a_i)^2\barrho$ is a sub-exponential random variable with $\psi_1$ norm $O(\alpha\normf{H}^2)$. We can use Lemma \ref{Bernstein inequality} to obtain that
\begin{equation}\label{eq:2le52}
\begin{aligned}
    \frac{1}{m}\sum_{i=1}^m(a_i^\T H\Xb^\T a_i)^2\barrho &\ge \E\big[(a_i^\T H\Xb^\T a_i)^2\barrho\big]-\delta\normf{H}^2 \\
   &\ge (0.78\sigma_r-\delta)\normf{H}^2+0.78\tr^2(H^\T\Xb)+0.78\normf{H^\T\Xb}^2
\end{aligned}
\end{equation}
holds with probability at least $1-\exp(-\Omega(\alpha^{-2}\delta^2m))$ where $\delta>0$. Combining (\ref{eq:1le52}) and (\ref{eq:2le52}), we obtain that (a) holds for a fixed $H\in \R^{n\times r}$.

 We construct an $\epsilon$-net $\mathcal{N}_\epsilon\subset \R^{n\times r}$ with cardinality $|\mathcal{N}_\epsilon|\le (1+\frac{2}{\epsilon})^{nr}$ such that for any $H\in\Rnr$ with $\normf{H}=1$, there exists $ H_0\in \mathcal{N}_\epsilon$ satisfying $\normf{H-H_0}\le \epsilon $. Taking a union bound over this set gives that
\begin{eqnarray*}
 \frac{1}{m}\sum_{i=1}^m(a_i^\T H_0\Xb^\T a_i)^2\barrho &\ge& (0.78\sigma_r-\delta)\normf{H_0}^2+0.78\tr^2(H_0^\T\Xb)+0.78\normf{H_0^\T\Xb}^2
\end{eqnarray*}
holds for all $H_0 \in \mathcal{N}_\epsilon$ with probability at least $1-(1+\frac{2}{\epsilon})^{nr}\exp(-\Omega(\alpha^{-2}\delta^2m))$. \\
Note that $\barrho<1$ for all $i$. Then there exists a universal constant $c_1>0$ such that
\begin{eqnarray}
    \left|\frac{1}{m}\sum_{i=1}^m (a_i^\T H\Xb^\T a_i)^2\barrho-\frac{1}{m}\sum_{i=1}^m (a_i^\T H_0\Xb^\T a_i)^2\barrho\right|
   &\le & \frac{1}{m}\sum_{i=1}^m\left|a_i^\T H\Xb^\T a_i-a_i^\T H_0\Xb^\T a_i\right|\nonumber \\
   &\le & c_1\|HX^\T-H_0X^\T\|_*  \nonumber\\
   &\le& c_1\sqrt{r}\normf{H-H_0} \nonumber\\
   & \le& c_1\sqrt{r}\epsilon  \label{invoking epsilon}
\end{eqnarray}
where we use Lemma \ref{lemma 3.1} in the second line, the fact $\|A\|_*\le \sqrt{\rank(A)}\normf{A}$ in the third line. Indeed, according to Lemma \ref{lemma 3.1}, for any $\delta\in(0,1)$, if $m\ge c_0\delta^{-2}n$, then with probability at least $1-C\exp(-\Omega(n))$ we have
\[ \frac{1}{m}\sum\limits_{i=1}^m\left|a_i^\T HX^\T a_i-a_i^\T H_0X^\T a_i\right| \le (1+\delta)\|HX^\T-H_0X^\T\|_* \le c_1\|HX^\T-H_0X^\T\|_*.
\]
By choosing $\epsilon=\frac{\delta}{c_1\sqrt{r}}$ in (\ref{invoking epsilon}), we conclude the first part of lemma.\\

We now turn to the part (b). The (\ref{ei}) implies that
\[
 \rho_{i,\alpha}\le \exp\big(-\frac{a_i^\T XX^\T a_i}{1.01\alpha}\big)
\]
holds with high probability.
It gives that
\begin{eqnarray*}
 \frac{1}{m}\sum_{i=1}^m(a_i^\T H\Xb^\T a_i)^2\rho_{i,\alpha} &\le & \frac{1}{m}\sum_{i=1}^m(a_i^\T H\Xb^\T a_i)^2\exp\left(-\frac{a_i^\T XX^\T a_i}{1.01\alpha}\right).
\end{eqnarray*}
From Claim \ref{claim1}, we have
\begin{eqnarray*}
 \E\left[(a_i^\T H\Xb^\T a_i)^2\exp\left(-\frac{a_i^\T XX^\T a_i}{1.01\alpha}\right)\right]&\le & \sigma_1\normf{H}^2+\tr^2(H^\T\Xb)+\normf{H^\T\Xb}^2.
\end{eqnarray*}
Similarly, $(a_i^\T H\Xb^\T a_i)^2\exp\big(-\frac{a_i^\T XX^\T a_i}{1.01\alpha}\big)$ is a sub-exponential random variable with sub-exponential norm $O(\alpha\normf{H}^2)$.
Then, we can employ the  method for proving part (a) to prove part (b).
\end{proof}

\begin{lemma}  \label{lemma 2}
For a fixed $\lambda>0$, for any $H\in \Rnr$ and $\delta>0$, if $m\ge c_0\delta^{-2}\lambda^{-2}nr\log(\sqrt{r}/(\delta\lambda))$, then with probability at least $1-C\exp(-\Omega(\delta^2\lambda^2m))$, we have
\[ \frac{1}{m}\sum_{i=1}^m(a_i^\T HH^\T a_i)^2\exp\left(-\lambda\frac{a_i^\T HH^\T a_i}{\normf{H}^2}\right)
\le 2\normf{HH^\T}^2+(2\delta+1)\normf{H}^4.
\]
Here, $c_0,C$ are some universal constants.
\end{lemma}
\begin{proof}
Without loss of generality, we only need to prove the lemma in the case $\normf{H}=1$. It is straightforward to show that
\begin{equation*}
  \E\left[(a_i^\T HH^\T a_i)^2\exp\left(-\lambda a_i^\T HH^\T a_i\right)\right] \le \E\left[(a_i^\T HH^\T a_i)^2\right]=2\normf{HH^\T}^2+\normf{H}^4.
\end{equation*}
Observe that $(a_i^\T HH^\T a_i)^2\exp\big(-\lambda a_i^\T HH^\T a_i\big)$ is a sub-exponential random variable with sub-exponential norm $O(1/\lambda\cdot\normf{H}^2)$. According to Lemma \ref{Bernstein inequality} we have
\begin{equation*}
 \frac{1}{m}\sum_{i=1}^m(a_i^\T HH^\T a_i)^2\exp\left(-\lambda a_i^\T HH^\T a_i\right)\le 2\normf{HH^\T}^2+\normf{H}^4+\frac{\delta_0}{\lambda}\normf{H}^2
\end{equation*}
with probability $1-\exp(-\Omega(\delta_0^2m))$. We next construct an $\epsilon$-net $\mathcal{N}_\epsilon$ with $|\mathcal{N}_\epsilon|\le (1+\frac{2}{\epsilon})^{nr}$ such that for any $H\in\Rnr$ with $\normf{H}=1$, there exists $H_0\in \mathcal{N}_\epsilon$ satisfying $\normf{H-H_0}\le \epsilon$. Since $x^2e^{-\lambda x}$ is Lipschitz function with Lipschitz constant $O(1/\lambda^2)$, we have
\[ \begin{array}{l}
     \Big|\frac{1}{m}\sum_{i=1}^m(a_i^\T HH^\T a_i)^2\exp\Big(-\lambda a_i^\T HH^\T a_i \Big)-\frac{1}{m}\sum_{i=1}^m(a_i^\T H_0H_0^\T a_i)^2\exp\Big(-\lambda a_i^\T H_0H_0^\T a_i\Big)\Big|  \vspace{1ex} \\
     \le \frac{1}{\lambda^2m}\sum_{i=1}^m\Big|a_i^\T HH^\T a_i-a_i^\T H_0H_0^\T a_i\Big| \vspace{1ex} \\
     \le \frac{c_2\sqrt{r}\epsilon}{\lambda^2}
   \end{array}
\]
where the last inequality follows from Lemma \ref{lemma 3.1}.  By choosing $\epsilon=\frac{\delta_0\lambda}{c_2\sqrt{r}}$, we obtain
\begin{equation*}
 \frac{1}{m}\sum_{i=1}^m(a_i^\T HH^\T a_i)^2\exp\left(-\lambda a_i^\T HH^\T a_i\right)\le 2\normf{HH^\T}^2+\normf{H}^4+\frac{2\delta_0}{\lambda}\normf{H}^2
\end{equation*}
with probability at least $1-\exp(-\Omega(\delta_0^2m))$ if $m\ge c_0\delta_0^{-2}nr\log(\sqrt{r}/(\delta_0\lambda))$. Finally, noting that $\normf{H}=1$ and taking $\delta_0=\lambda\delta$, we arrive at the conclusion.
\end{proof}
\begin{corollary} \label{ahhalow}
For any $\delta>0$, $U\in \Rnr$ and $H=U-\Xb$, if $m\ge c_0\alpha^2\delta^{-2}\sigma_r^{-2}nr\log(\alpha\sqrt{r}/(\delta\sigma_r))$, then with probability at least $1-C\exp(-\Omega(n))$, it holds
\begin{equation*}
  \frac{1}{m}\sum_{i=1}^m(a_i^\T HH^\T a_i)^2\rho_{i,\alpha}\le 2\normf{HH^\T}^2+(2\delta+1)\normf{H}^4.
\end{equation*}
\end{corollary}
\begin{proof}
Since $\sigma_r$ is the smallest eigenvalue of $XX^\T$, we have
\begin{equation*}
y_i=a_i^\T XX^\T a_i\ge \sigma_r\norms{a_i}^2,
\end{equation*}
which implies that
\begin{equation}\label{eq:budeng1}
\norms{a_i}^2\le \frac{a_i^\T XX^\T a_i}{\sigma_r}=\frac{y_i}{\sigma_r}.
\end{equation}
On the other hand,  we have
\begin{equation}\label{eq:budeng2}
a_i^\T HH^\T a_i\leq \|H\|_F^2\|a_i\|^2.
\end{equation}
Combining (\ref{eq:budeng1}) and (\ref{eq:budeng2}), we obtain that
\begin{equation}\label{eq:budeng3}
y_i\geq \sigma_r \frac{a_i^\T HH^\T a_i}{\|H\|_F^2}.
\end{equation}
According to  (\ref{ei}) and (\ref{eq:budeng3}), we obtain that
\begin{eqnarray*}
   \frac{1}{m}\sum_{i=1}^m(a_i^\T HH^\T a_i)^2\rho_{i,\alpha} &\le& \frac{1}{m}\sum_{i=1}^m(a_i^\T HH^\T a_i)^2\exp\left(-\frac{\sigma_r}{1.01\alpha}\cdot\frac{a_i^\T HH^\T a_i}{\normf{H}^2}\right).
\end{eqnarray*}
We take $\lambda=\frac{\sigma_r}{1.01\alpha}$ in Lemma \ref{lemma 2}  and arrive at the conclusion.
\end{proof}

\begin{proof}[Proof of Proposition \ref{pr:1} ]
To state conveniently, we set
\begin{equation*}
\beta^2 =\frac{1}{m}\sum_{i=1}^m(a_i^\T HH^\T a_i)^2 \rho_{i,\alpha},\quad
\gamma^2=\frac{2}{m}\sum_{i=1}^m(a_i^\T H\Xb^\T a_i)^2 \rho_{i,\alpha}.
\end{equation*}
According to the expression of exponential-type gradient (\ref{regradient}), we have
\begin{equation}
\begin{aligned}
  \langle\nabla f_{\ex}(U),H\rangle \nonumber
   &=\beta^2+\gamma^2+\frac{3}{m}\sum\limits_{i=1}^m(a_i^\T H\Xb^\T a_i)(a_i^\T HH^\T a_i)\rho_{i,\alpha}\nonumber \\
   &\ge  \beta^2+\gamma^2-\frac{3}{m}\sqrt{\sum\limits_{i=1}^m(a_i^\T H\Xb^\T a_i)^2 \rho_{i,\alpha}}\cdot
  \sqrt{\sum\limits_{i=1}^m(a_i^\T HH^\T a_i)^2 \rho_{i,\alpha}} \nonumber\\
   &= \beta^2+\gamma^2-\frac{3}{\sqrt{2}} \beta \gamma=\Big(\gamma-\frac{3}{2\sqrt{2}}\beta\Big)^2-\frac{1}{8}\beta^2 \nonumber\\
   &\ge  \Big(\frac{\gamma^2}{2}-\frac{9}{8}\beta^2\Big)-\frac{1}{8}\beta^2=\frac{\gamma^2}{2}-\frac{5}{4}\beta^2 \nonumber\\
   &= \frac{1}{m}\sum_{i=1}^m(a_i^\T H\Xb^\T a_i)^2 \rho_{i,\alpha}
  -\frac{5}{4m}\sum_{i=1}^m(a_i^\T HH^\T a_i)^2 \rho_{i,\alpha} \nonumber\\
   &\ge  (0.78\sigma_r-2\delta_1)\normf{H}^2+0.78\tr^2(H^\T\Xb)+0.78\normf{H^\T\Xb}^2-\frac{5}{2}\normf{HH^\T}^2-\frac{5(2\delta_2+1)}{4}\normf{H}^4 \nonumber\\
   &\ge  \left(0.78\sigma_r-2\delta_1-\frac{5(2\delta_2+3)}{4}\normf{H}^2\right)\normf{H}^2+0.78\left(\tr^2(H^\T\Xb)+\normf{H^\T\Xb}^2\right)  \label{choice of delta_1 delta_2}
\end{aligned}
\end{equation}
where we use Cauchy-Schwarz inequality in the second line, the inequality $(\gamma-\beta)^2\ge\frac{\gamma^2}{2}-\beta^2$ in the fourth line, Lemma \ref{lemma one} and Corollary \ref{ahhalow} in the sixth line, and the fact that $\normf{HH^\T}\le \normf{H}^2$ in the last line.
Note that $\normf{H}^2=\|U-\Xb\|_F^2=d(U)^2\le\frac{1}{8}\sigma_r$.
 Taking $\delta_1\le \frac{1}{16}\sigma_r$ and $\delta_2\le \frac{1}{16}$, we obtain that
\begin{equation*}
\langle\nabla f_{\ex}(U),H\rangle\ge 0.166\sigma_r\normf{H}^2+0.78\left(\tr^2(H^\T\Xb)+\normf{H^\T\Xb}^2\right)
\end{equation*}
with probability at least $1-C\exp(-\Omega(n)$, if $m\ge c_0\sigma_r^{-2}nr\log(c_1r/\sigma_r)$. This implies the part $(a)$.\\

Next, we turn to the part $(b)$. We consider
\[
\normf{\nabla f_{\ex}(U)}^2=\maxm{\normf{W}=1, W\in \R^{n\times r}}\abs{\langle\nabla f_{\ex}(U),W\rangle}^2
\]
on the case where $H=U-\Xb\le\sqrt{\frac{1}{8}\sigma_r}$. Recall the notation $\rho_{i,\alpha}$ in formula (\ref{hi}), and we have
\begin{equation*}
\begin{aligned}
  &\abs{\langle\nabla f_{\ex}(U),W\rangle}^2\\
   &= \left(\frac{1}{m}\sum_{i=1}^m(a_i^\T HH^\T a_i)(a_i^\T HW^\T a_i)\rho_{i,\alpha}+\frac{2}{m}\sum_{i=1}^m(a_i^\T H\Xb^\T a_i)(a_i^\T HW^\T a_i)\rho_{i,\alpha} \right.\\
   &\quad \left.+\frac{1}{m}\sum_{i=1}^m(a_i^\T HH^\T a_i)(a_i^\T \Xb W^\T a_i)\rho_{i,\alpha}+\frac{2}{m}\sum_{i=1}^m(a_i^\T H\Xb^\T a_i)(a_i^\T \Xb W^\T a_i)\rho_{i,\alpha} \right)^2 \\
   &\le  4\left(\frac{1}{m}\sum_{i=1}^m(a_i^\T HH^\T a_i)(a_i^\T HW^\T a_i)\rho_{i,\alpha}\right)^2+16\left(\frac{1}{m}\sum_{i=1}^m(a_i^\T H\Xb^\T a_i)(a_i^\T HW^\T a_i)\rho_{i,\alpha}\right)^2 \\
   &\quad +4\left(\frac{1}{m}\sum_{i=1}^m(a_i^\T HH^\T a_i)(a_i^\T \Xb W^\T a_i)\rho_{i,\alpha}\right)^2+16\left(\frac{1}{m}\sum_{i=1}^m(a_i^\T H\Xb^\T a_i)(a_i^\T \Xb W^\T a_i)\rho_{i,\alpha}\right)^2.
\end{aligned}
\end{equation*}
 We first consider the term $4\left(\frac{1}{m}\sum_{i=1}^m(a_i^\T HH^\T a_i)(a_i^\T HW^\T a_i)\rho_{i,\alpha}\right)^2$. Using Cauchy-Schwarz inequality, we obtain that
\begin{equation*}
\begin{aligned}
& 4\left(\frac{1}{m}\sum\limits_{i=1}^m(a_i^\T HH^\T a_i)(a_i^\T HW^\T a_i)\rho_{i,\alpha}\right)^2\\
   &\le 4\left(\frac{1}{m}\sum\limits_{i=1}^m(a_i^\T HH^\T a_i)^2\rho_{i,\alpha}\right)
    \left(\frac{1}{m}\sum\limits_{i=1}^m(a_i^\T HW^\T a_i)^2\rho_{i,\alpha}\right) \\
    &\le 4\left(\frac{1}{m}\sum\limits_{i=1}^m(a_i^\T HH^\T a_i)^2\rho_{i,\alpha}\right)\left(\frac{1}{m}\sum\limits_{i=1}^m(a_i^\T HH^\T a_i)(a_i^\T WW^\T a_i)\rho_{i,\alpha}\right).
\end{aligned}
\end{equation*}
According to Corollary \ref{ahhalow}, we have
\begin{equation}\label{eq:T_1}
  \frac{1}{m}\sum\limits_{i=1}^m(a_i^\T HH^\T a_i)^2\rho_{i,\alpha} \le\left(2\normf{HH^\T}^2+(2\delta_2+1)\normf{H}^4\right)
\end{equation}
with probability at least $1-C\exp(-\Omega(n))$ provided $m\ge c_0\delta_2^{-2}\sigma_r^{-2}nr\log(\sqrt{r}/(\delta_2\sigma_r))$.  Noting that $a_i^\T XX^\T a_i\ge \sigma_r\norms{a_i}^2$ and $a_i^\T HH^\T a_i\leq \|H\|_F^2 \|a_i\|^2$ we have
\begin{equation*}
  \frac{a_i^\T XX^\T a_i}{2.02\alpha}\ge \frac{\sigma_r \cdot a_i^\T HH^\T a_i}{2.02\alpha\normf{H}^2} \qquad \text{and}\qquad
  \frac{a_i^\T XX^\T a_i}{2.02\alpha}\ge \frac{\sigma_r \cdot a_i^\T WW^\T a_i}{2.02\alpha}.
\end{equation*}
It gives that
\begin{equation} \label{eq:T_2}
\begin{aligned}
   &(a_i^\T HH^\T a_i)(a_i^\T WW^\T a_i)\rho_{i,\alpha} \\
   &\le  (a_i^\T HH^\T a_i)(a_i^\T WW^\T a_i)\exp\left(-\frac{a_i^\T XX^\T a_i}{1.01\alpha}\right)\\
   &\le (a_i^\T HH^\T a_i)\exp\left(-\frac{\sigma_r \cdot a_i^\T HH^\T a_i}{2.02\alpha\normf{H}^2}\right)
   (a_i^\T WW^\T a_i)\exp\left(-\frac{\sigma_r \cdot a_i^\T WW^\T a_i}{2.02\alpha}\right) \\
   &\le \normf{H}^2 \Big(\frac{1.01\alpha}{e\sigma_r}\Big)^2
\end{aligned}
\end{equation}
where we use inequality $ xe^{-\gamma x}\le 1/(e\gamma)$ for any $x\ge 0$ in the last line. Combining formulas (\ref{eq:T_1}) and (\ref{eq:T_2}), we obtain
\begin{equation*}
  4\left(\frac{1}{m}\sum\limits_{i=1}^m(a_i^\T HH^\T a_i)(a_i^\T HW^\T a_i)\rho_{i,\alpha}\right)^2\le 4 \Big(\frac{1.01\alpha}{e\sigma_r}\Big)^2 \normf{H}^2 \left(2\normf{HH^\T}^2+(2\delta_2+1)\normf{H}^4\right).
\end{equation*}
The other three terms can be bounded similarly. For the second term, we have
\begin{equation*}
\begin{aligned}
   &16\left(\frac{1}{m}\sum_{i=1}^m(a_i^\T H\Xb^\T a_i)(a_i^\T HW^\T a_i)\rho_{i,\alpha}\right)^2 \\
   &\le 16\left(\frac{1}{m}\sum\limits_{i=1}^m(a_i^\T H\Xb^\T a_i)^2\rho_{i,\alpha}\right)\left(\frac{1}{m}\sum\limits_{i=1}^m(a_i^\T HW^\T a_i)^2\rho_{i,\alpha}\right) \\
   &\le 4 \Big(\frac{1.01\alpha}{e\sigma_r}\Big)^2 \normf{H}^2 \left(4(\sigma_1+2\delta_1)\normf{H}^2+4\tr^2(H^\T\Xb)+4\normf{H^\T\Xb}^2\right)
\end{aligned}
\end{equation*}
with probability at least $1-C\exp(-\Omega(n))$ provided $m\ge c_0\delta_1^{-2}nr\log(\sqrt{r}/\delta_1)$, where we use the part (b) of Lemma \ref{lemma one} in the last line. The third term and fourth term can be bounded as
\begin{equation*}
  4\left(\frac{1}{m}\sum_{i=1}^m(a_i^\T HH^\T a_i)(a_i^\T \Xb W^\T a_i)\rho_{i,\alpha}\right)^2 \le
  4\Big(\frac{1.01\alpha}{e\sigma_r}\Big)^2\normf{X}^2\left(2\normf{HH^\T}^2+(2\delta_2+1)\normf{H}^4\right)
\end{equation*}
\begin{equation*}
  16\left(\frac{1}{m}\sum_{i=1}^m(a_i^\T H\Xb^\T a_i)(a_i^\T \Xb W^\T a_i)\rho_{i,\alpha}\right)^2 \le
   4 \Big(\frac{1.01\alpha}{e\sigma_r}\Big)^2\normf{X}^2\left(4(\sigma_1+2\delta_1)\normf{H}^2+4\tr^2(H^\T\Xb)+4\normf{H^\T\Xb}^2\right).
\end{equation*}
Putting there inequalities together and noting that $\normf{HH^\T}\le \normf{H}^2$, we have
\[\normf{\nabla f_{\ex}(U)}^2 \le 4 \Big(\frac{1.01\alpha}{e\sigma_r}\Big)^2\left(\normf{H}^2+\normf{X}^2\right)
\left(\left(4\sigma_1+8\delta_1+(2\delta_2+3)\normf{H}^2\right)\normf{H}^2+4\tr^2(H^\T\Xb)+4\normf{H^\T\Xb}^2\right).
\]
Furthermore, noticing that $\normf{H}^2\le\frac{1}{8}\sigma_r$ and choosing $\delta_1\le \frac{1}{16}\sigma_r$, $\delta_2\le \frac{1}{16}$, it follows that
\begin{equation*}
\frac{\sigma_r^2\normf{\nabla f_{\ex}(U)}^2}{3\alpha^2\left(\normf{H}^2+\normf{X}^2\right)}
\le 1.223\sigma_1\normf{H}^2+\tr^2(H^\T\Xb)+\normf{H^\T\Xb}^2
\end{equation*}
with probability at least $1-C\exp(-\Omega(n)$, if $m\ge c_0\sigma_r^{-2}nr\log(c_1r/\sigma_r)$.
\end{proof}

The rest paper is to check the Claim \ref{claim1}. For 1) and 2) of the Claim \ref{claim1}, let $O_1=\argmin{O \in \mathcal{O}(r)}\normf{U-XO}$, then $\Xb=XO_1$. Recall that $X$ has orthogonal column vectors, and then there exists an orthogonal matrix $O_2 \in \R^{n\times n}$ such that $X=O_2(\norms{x_1}e_1,\ldots,\norms{x_r}e_r)$. Let $\Hh:=HO_1^\T, \; \Ht=O_2^\T\Hh $ and $\hh_s,\htt_s,x_s$ denote the $s$th column of $\Hh,\Ht,X$ respectively, and $a_{i,s}$ denotes the $s$th entry of $a_i$. It follows that
\begin{eqnarray}
   & & \E\left[(a_i^\T H\Xb^\T a_i)^2\right]=\E\left[(a_i^\T\Hh X^\T a_i)^2\right]=\E (a_i^\T O_2 \Ht X^\T O_2O_2^\T a_i) \nonumber \\
& \quad &=\E (a_i^\T  \Ht X^\T O_2 a_i)   = \E\left[\|x_1\|(\htt_1^\T a_i)a_{i,1}+\cdots+\|x_r\|(\htt_r^\T a_i)a_{i,r}\right]^2 \nonumber \\
   &\quad &= \E\left[\sum\limits_{s=1}^r\|x_s\|^2(\htt_s^\T a_i)^2a_{i,s}^2+\sum\limits_{s\neq k}\|x_s\|\|x_k\|(\htt_s^\T a_i)(\htt_k^\T a_i)a_{i,s}a_{i,k}\right] \nonumber\\
   &\quad &= \sum\limits_{s=1}^r\left(\norms{x_s}^2\norms{\htt_s}^2+2\norms{x_s}^2\htt_{s,s}^2\right)+\sum\limits_{s\neq k}\norms{x_s}\norms{x_k}\left(\htt_{s,s}\htt_{k,k}+\htt_{s,k}\htt_{k,s}\right)\label{firtm} \\
   &\quad &= \sum\limits_{s=1}^r\norms{x_s}^2\norms{\hh_s}^2+\sum\limits_{s,k}\norms{x_s}\norms{x_k}\left(\htt_s^\T e_s\htt_k^\T e_k+\htt_s^\T e_k\htt_k^\T e_s\right)\nonumber\\
   &\quad &= \sum\limits_{s=1}^r\norms{x_s}^2\norms{\hh_s}^2+\sum\limits_{s,k}(x_s^\T\hh_sx_k^\T\hh_k+x_s^\T\hh_kx_k^\T\hh_s)\label{The max sigular}  \\
   &\quad &\ge \sigma_r\normf{\Hh}^2+\tr^2(X^\T\Hh)+\tr(X^\T\Hh X^\T\Hh)  \nonumber \\
   &\quad&= \sigma_r\normf{H}^2+\tr^2(H^\T\Xb)+\tr(H^\T\Xb H^\T\Xb)   \nonumber \\
   &\quad &= \sigma_r\normf{H}^2+\tr^2(H^\T\Xb)+\normf{H^\T\Xb}^2,\nonumber
\end{eqnarray}
where the last equation follows from that $H^\T\Xb$ is a symmetric matrix and the symmetry of $HX^\T=(U-\Xb)X^\T$ can be seen by the singular-value decomposition of  $X^\T U$. More specifically, suppose that the singular-value decomposition of $X^\T U$ is $WDV^\T$, then we have
\begin{align*}
 O_1:=\argmin{O \in \mathcal{O}(r)}\normf{U-XO} & =\argmax{O \in \mathcal{O}(r)}\langle XO,U\rangle =\argmax{O \in \mathcal{O}(r)}\langle O,WDV^\T\rangle=WV^\T.
\end{align*}
Therefore, $U^\T\Xb=U^\T XWV^\T=VDV^\T$ is a symmetric matrix, which implies that $H^\T\Xb=U^\T\Xb-\Xb^\T\Xb$ is also symmetric matrix.\\

Similarly, from formula (\ref{The max sigular}), it is easy to obtain
\[\E\left[(a_i^\T H\Xb^\T a_i)^2\right]\le\sigma_1\normf{H}^2+\tr^2(H^\T\Xb)+\normf{H^\T\Xb}^2.\]

For 3) of the Claim \ref{claim1}, using the notation $\Hh,\Ht,\hh_s,\htt_s$ above, we have
\begin{eqnarray*}
 \E\left[(a_i^\T H\Xb^\T a_i)^2\barrho\right] &=& \E\left[(a_i^\T\Hh X^\T a_i)^2\barrho \right] \\
   &=& \E\left[\sum\limits_{s=1}^r\|x_s\|^2(\htt_s^\T a_i)^2a_{i,s}^2\cdot\prod\limits_{t=1}^r\exp\left(-\frac{\norms{x_t}^2a_{i,t}^2}{0.99\alpha}\right)\right]\\
   & &+\E\left[\sum\limits_{s\neq k}\|x_s\|\|x_k\|(\htt_s^\T a_i)(\htt_k^\T a_i)a_{i,s}a_{i,k}\cdot
    \prod\limits_{t=1}^r\exp\left(-\frac{\norms{x_t}^2a_{i,t}^2}{0.99\alpha}\right)\right] \\
   &>& 0.78\sum\limits_{s=1}^r\|x_s\|^2(2\htt_{s,s}^2+\norms{\htt_{s}}^2)+0.78\sum\limits_{s\neq k}\norms{x_s}\norms{x_k}(\htt_{s,s}\htt_{k,k}+\htt_{s,k}\htt_{k,s})  \\
   &=& 0.78\E\left[(a_i^\T H\Xb^\T a_i)^2\right],
\end{eqnarray*}
where the last equation follows from (\ref{firtm}) and the inequality comes from the following two inequalities (\ref{ineqality:1}) and (\ref{inequality:2}):
\begin{eqnarray}
   && \E\Big[(\htt_s^\T a_i)^2a_{i,s}^2\cdot\prod\limits_{t=1}^r\exp(-\frac{\norms{x_t}^2a_{i,t}^2}{0.99\alpha})\Big] \nonumber\\
  &= & \frac{1}{\gamma\omega_s}\Big(\frac{\htt_{s,1}^2}{\omega_1}+\cdots+\frac{\htt_{s,s-1}^2}{\omega_{s-1}}
  +\frac{3\htt_{s,s}^2}{\omega_s}+\frac{\htt_{s,s+1}^2}{\omega_{s+1}}+\cdots+\frac{\htt_{s,r}^2}{\omega_r}+\htt_{s,r+1}^2+\cdots+\htt_{s,n}^2\Big)\nonumber \\
  & \ge&\frac{1}{1.102^2\cdot\gamma}(\htt_{s,1}^2+\cdots+\htt_{s,s-1}^2+3\htt_{s,s}^2+\htt_{s,s+1}^2+\cdots+\htt_{s,n}^2)\nonumber\\
  & \ge&\frac{1}{1.102^2\cdot e^{1/0.99\alpha}}(2\htt_{s,s}^2+\norms{\htt_{s}}^2)\nonumber\\
  &>&0.78(2\htt_{s,s}^2+\norms{\htt_{s}}^2)\label{ineqality:1}
\end{eqnarray}
provided $\alpha \ge 20$ and the parameters $\omega_k,\gamma$ are defined as follows:
 $$\omega_k:=\frac{\|x_k\|^2}{0.495\alpha}+1\le 1.102,\;\forall\; 1\le k\le r$$
 and
 $$\gamma:=\sqrt{(\frac{\|x_1\|^2}{0.495\alpha}+1)(\frac{\|x_2\|^2}{0.495\alpha}+1)\cdots(\frac{\|x_r\|^2}{0.495\alpha}+1)}\le e^{1/0.99\alpha}$$
due to the fact that $1+x\le e^x$ for any $x\ge 0$ and $\normf{X}=1$. Similarly, for any $s\neq k,1\le s,k\le r$, we have
\begin{equation}
   \E\left[(\htt_s^\T a_i)(\htt_k^\T a_i)a_{i,s}a_{i,k}\cdot\prod\limits_{t=1}^r\exp(-\frac{\norms{x_t}^2a_{i,t}^2}{0.99\alpha})\right]= \frac{\htt_{s,s}\htt_{k,k}+\htt_{s,k}\htt_{k,s}}{\gamma\omega_s\omega_k}> 0.78(\htt_{s,s}\htt_{k,k}+\htt_{s,k}\htt_{k,s}).\label{inequality:2}
\end{equation}

\end{document}